\documentclass[pdflatex,sn-mathphys-num]{sn-jnl}

\usepackage{graphicx}%
\usepackage{multirow}%
\usepackage{amsmath,amssymb,amsfonts}%
\usepackage{amsthm}%
\usepackage{mathrsfs}%
\usepackage[title]{appendix}%
\usepackage{xcolor}%
\usepackage{textcomp}%
\usepackage{manyfoot}%
\usepackage{booktabs}%
\usepackage{algorithm}%
\usepackage{algorithmicx}%
\usepackage{algpseudocode}%
\usepackage{listings}%

\theoremstyle{thmstyleone}%
\newtheorem{theorem}{Theorem}[section]
\newtheorem{proposition}[theorem]{Proposition}%
\newtheorem{lemma}[theorem]{Lemma}
\newtheorem{corollary}[theorem]{Corollary}

\theoremstyle{thmstyletwo}%
\newtheorem{example}[theorem]{Example}%
\newtheorem{remark}[theorem]{Remark}%

\theoremstyle{thmstylethree}%
\newtheorem{definition}[theorem]{Definition}%

\numberwithin{equation}{section}

\begin{document}

\title[A new complete proof of the random Brouwer fixed point theorem and its implied consequences of unification]{A new complete proof of the random Brouwer fixed point theorem and its implied consequences of unification}


\author[1]{\fnm{Qiang} \sur{Tu}}\email{qiangtu126@126.com}

\author[1]{\fnm{Xiaohuan} \sur{Mu}}\email{xiaohuanmu@163.com}

\author*[1]{\fnm{Tiexin} \sur{Guo}}\email{tiexinguo@csu.edu.cn}

\footnotetext{This work was supported by the National Natural Science Foundation of China
	(Grant No.12371141) and the Natural Science Foundation of Hunan Province of
	China (Grant No.2023JJ30642).}

\author[2]{\fnm{Goong} \sur{Chen}}\email{gchen@math.tamu.edu} 

\affil[1]{School of Mathematics and Statistics, Central South University, Changsha 410083, China}

\affil[2]{Department of Mathematics and Institute for Quantum Science and Engineering, Texas A\&M University, College Station, TX 77843, USA}


\abstract{
   We first establish a general random Sperner lemma by presenting a completely new approach for the theory of $L^{0}$-simplicial subdivisions of  $L^{0}$-simplexes. Based on this, we are able to achieve a new complete proof of the random Brouwer fixed theorem in random Euclidean spaces, which can provide a solid foundation for various contemporary applications of interest. Afterward, we unify the works currently available and closely related to the random Brouwer fixed theorem: we first prove that the stochastic Brouwer fixed point theorem occurring elsewhere in stochastic analysis is equivalent to a special  case of our random Brouwer fixed theorem, and then prove a general random Borsuk theorem and its equivalence with the random Brouwer fixed theorem. Finally, we conclude this paper with commentaries on recent state of study of the famous Schauder conjecture. 
	}




\maketitle

\tableofcontents

\section{Introduction}\label{sec.1}

\par 
The celebrated Brouwer fixed point theorem \cite{Brouwer1912} states that every continuous map from a bounded closed convex subset of $\mathbb{R}^{d}$ to itself has a fixed point. Here and throughout, $\mathbb{R}^{d}$ denotes the $d$-dimensional Euclidean space. The Brouwer fixed point theorem  has numerous applications in analysis, geometry, differential equations and mathematical economics \cite{Mawhin2020,CM2023,Feltrin2022,Nash1951}. Furthermore, it is also attractive that Brouwer fixed point theorem is equivalent to many prominent fundamental results \cite{Mawhin2019,IKM2014,IKM2021,Park1999}, for example, the Knaster-Kuratowski-Mazurkiewicz (briefly, $KKM$) mapping principle \cite{KKM1929}, the classical Borsuk theorem \cite{Borsuk1967} and the Gustafson-Schmitt-Cid-Mawhin theorem on the existence of periodic solutions of ordinary differential equations \cite[Theorem 1.1]{CM2023}. Due to its fundamental importance \cite{Cas1996}, the Brouwer fixed point theorem has been given many proofs \cite{Mawhin2020,Park1999}. One of the most elegant proofs of all of its classical proofs is via the Sperner lemma \cite{Sperner1928}, in particular that proof also helps in developing a simplicial approximation of fixed points \cite{Kuhn1968,Sca1967}. Further, the Brouwer fixed point theorem has been generalized in various ways \cite{Kakutani1941,Sch1930,Ty1935} and has had a major impact on the development of nonlinear analysis and the general fixed point theory \cite{Bor1985,DM2021,GD2003,Cauty2007,ET2021,Yu2024}.

\par With an aim of developing random functional analysis \cite{Guo2010,GWXYC2025}, dynamic Nash equilibrium theory \cite{DKKS2013,GWXYC2025} and stochastic differential equations \cite{Ponosov1987,Ponosov1988,Ponosov2021,Ponosov2022}, there have been strong desires to extend the Brouwer fixed point theorem to the setting of a random Euclidean space $L^{0}(\mathcal{F},\mathbb{R}^{d})$, where $L^{0}(\mathcal{F},\mathbb{R}^{d})$ is the linear space of equivalence
classes of random vectors from a probability space $(\Omega,\mathcal{F},P)$ to the $d$-dimensional Euclidean space $\mathbb{R}^{d}$.
Since $L^{0}(\mathcal{F},\mathbb{R}):=L^{0}(\mathcal{F},\mathbb{R}^{1})$ is an algebra over the real number field $\mathbb{R}$,
$L^{0}(\mathcal{F},\mathbb{R}^{d})$ is a free module of rank $d$ over the algebra $L^{0}(\mathcal{F},\mathbb{R})$, and thus the notions of an
$L^{0}$-convex set and a $\sigma$-stable set can be introduced and developed. More importantly, $L^{0}(\mathcal{F},\mathbb{R}^{d})$ can be 
regarded as a special complete random normed module. The notion of a random normed module was introduced independently by Guo \cite{Guo1992,Guo1996,Guo1999,Guo2010,GMT2024,GWXYC2025} in connection
with the idea of random-structured space theory and by Gigli \cite{Gigli2018} in connection with nonsmooth differential 
geometry on metric measure spaces (see also \cite{BPS2023,CLPV2025,CGP2025,GLP2025,LP2019,LPV2024} for related advances). When regarded 
as a random normed module, $L^{0}(\mathcal{F},\mathbb{R}^{d})$ is often endowed with two topologies: one is the $(\varepsilon,\lambda)$-topology, 
denoted by $\mathcal{T}_{\varepsilon,\lambda}$, which coincides with the usual topology of convergence in probability and is a typical metrizable locally nonconvex 
linear topology and makes $(L^{0}(\mathcal{F},\mathbb{R}^{d}),\mathcal{T}_{\varepsilon,\lambda})$ a topological module over the topological algebra 
$(L^{0}(\mathcal{F},\mathbb{R}),\mathcal{T}_{\varepsilon,\lambda})$, while the other is the locally $L^{0}$-convex topology, denoted by $\mathcal{T}_{c}$, which 
was introduced in \cite{FKV2009}, coincides with the topology generated by random open balls, and makes $(L^{0}(\mathcal{F},\mathbb{R}^{d}),\mathcal{T}_{c})$ a topological module over the topological ring $(L^{0}(\mathcal{F},\mathbb{R}),\mathcal{T}_{c})$; see \cite{Guo2010} for the  inherent connections between the two kinds of topologies. Thus, according to different choices of $\mathcal{T}_{\varepsilon,\lambda}$ and $\mathcal{T}_{c}$, a random Brouwer fixed point theorem has possibly two corresponding formulations: the first is the formulation of random Brouwer fixed point theorem equipped with $\mathcal{T}_{c}$, which was first given in \cite{GWXYC2025} and is stated as Theorem \ref{thm.Brouwer} in this paper, while the second is the formulation with $\mathcal{T}_{\varepsilon,\lambda}$, which was first given in \cite{DKKS2013,GZWW2021} and is stated as Proposition \ref{prop.2.3} in this paper. It follows immediately from Lemma 4.3 of \cite{GWXYC2025} that Proposition \ref{prop.2.3}, namely, the $\mathcal{T}_{\varepsilon,\lambda}$-version of a random Brouwer fixed point theorem is a special case of Theorem \ref{thm.Brouwer} (namely, the $\mathcal{T}_{c}$-version of random Brouwer fixed point theorem). Hence, from now on, Theorem \ref{thm.Brouwer} will be called by us the random Brouwer fixed point theorem.

\par Just as the Brouwer fixed point theorem is the cornerstone of the topological fixed point theory, the random Brouwer fixed point theorem is of fundamental importance, for example, it has been used in \cite{GWXYC2025} for the proofs of both a noncompact Schauder fixed point theorem in random normed modules and the existence theorem of a conditional Nash equilibrium point. In particular, it will also be used to solve the open problems posed in \cite{GWXYC2025}, for example, in a forthcoming work where we will establish a noncompact $KKM$ mapping principle and further prove a general dynamic Nash equilibrium theorem while we will also provide an effective approximation algorithm for searching the dynamic Nash equilibrium point. The original proof of the random Brouwer fixed point theorem \cite{GWXYC2025} essentially employed a special random Sperner lemma (see the first part of \cite[Theorem 2.3]{DKKS2013}). However, the special random Sperner lemma has obvious limitations in its applications, and its proof given in \cite{DKKS2013} is also inaccessible to the reader, which renders our original proof of the random Brouwer fixed point theorem in \cite{GWXYC2025} neither complete nor satisfactory. Therefore, one of the central purposes of this paper is to give a new complete proof of the random Brouwer fixed point theorem. For this sake, in the frontal part of this paper we will first provide a new foundation (namely, a general random Sperner lemma) for such a complete proof.

\par  To emphasize the necessity of reconsidering the problem of randomizing the classical Sperner lemma, let us first recall the classical Sperner lemma and its remarkable applications, and give a brief comment on the previous work \cite{DKKS2013} on a special random Sperner lemma. In 1928, this classical lemma was given by Sperner \cite{Sperner1928}, also stated as Lemma \ref{lemm.Sperner} in this paper for the reader's convenience. Four decades later in 1968 an elegant proof was given by Kuhn \cite{Kuhn1968} via a graph-theoretic argument. As has been mentioned above, the Sperner lemma-based proof of the classical Brouwer fixed point theorem is significant since it has the following two merits. One is its elegance, the other is the constructiveness that leads to a simplicial approximation of fixed points \cite{Sca1967,Kuhn1968}. Mainly motivated by the proof of the classical Brouwer fixed point theorem, Drapeau et al. \cite{DKKS2013} first gave a special random Sperner lemma (see the first part of \cite[Theorem 2.3]{DKKS2013}) and further proved a special case of the random Brouwer fixed point theorem on an $L^{0}$-simplex (see the second part of \cite[Theorem 2.3]{DKKS2013}). Just as stated above, the classical Sperner lemma is so fundamental and important that it merits a complete random version, but the related, existing work in \cite{DKKS2013} has certain significant limitations, listed as follows:
\begin{enumerate}
	\item [(1)] Concerning the $L^{0}$-simplicial subdivision of an $L^{0}$-simplex,  limited by their method authors of \cite{DKKS2013} only considered the case of the $L^{0}$-barycentric subdivision so that  the special random Sperner lemma did not encompass the classical Sperner lemma as a special case. 
	\item [(2)] The $L^{0}$-barycentric subdivision is too special to ensure that the higher-fold $L^{0}$-barycentric subdivision of an $L^{0}$-simplex is again an $L^{0}$-barycentric subdivision so that their notion of an $L^{0}$-labeling function \cite[Definition 2.1]{DKKS2013} is not as natural as that of the usual labeling function.  When the special random Sperner lemma was applied to the proof of their special random Brouwer fixed point theorem on an $L^{0}$-simplex \cite[Theorem 2.3]{DKKS2013} the so obtained sequence of completely labeled $L^{0}$-simplexes approaching the desired fixed point is not decreasing, which rules out the possibility of designing an effective algorithm for the simplicial approximation of fixed points. 
	\item [(3)] Their proof of the special random Sperner lemma was similar to that of Kuhn \cite{Kuhn1968} for the classical Sperner lemma, but authors of \cite{DKKS2013} did not adopt a graph-theoretic argument and, thus, their proof was rather complicated, lacking readability as well as verifiability.
\end{enumerate}

\par To establish a more appeasing random Sperner lemma in order to overcome the limitations of the earlier approach , we first present a completely new approach for the theory of $L^{0}$-simplicial subdivisions of $L^{0}$-simplexes. In Sections \ref{sec.2}, \ref{sec.3} and \ref{sec.4}, we introduce a general notion of an $L^{0}$-simplicial subdivision of an $L^{0}$-simplex in connection with a usual simplicial subdivision of a classical simplex (see Definition \ref{defn.$L^{0}$-simplicial subdivision}), which helps us establish a key representation theorem of a proper $L^0$-labeling function by usual proper labeling functions (see Theorem \ref{thm.representation}). In turn, this representation theorem can help build a general random Sperner lemma in a concise way as a consequence of the classical Sperner lemma, where our success is achieved through the connection between proper $L^0$-labeling functions and usual proper labeling functions rather than by considering the classification of $L^{0}$-subsimplexes as in the proof of the special random Sperner lemma \cite{DKKS2013}. The general random Sperner lemma not only includes the classical Sperner lemma as a special case but also meets all the needs of our future applications. Further, by combining the general random Sperner lemma, Lemma 4.3 of \cite{GWXYC2025} and the randomized Bolzano-Weierstrass theorem \cite{KS2001} we are able to give a new complete proof of the random Brouwer fixed theorem in random Euclidean spaces and, in particular, we note that the proof contains constructive information that is quite desirable (see Remark \ref{rmk.Brouwer1}).

\par Recently, we have known from \cite{Ponosov1988,Ponosov2021,Ponosov2022}  that Ponosov has been working on a fixed point theory for a class of local and continuous-in-probability operators with 
the aim of providing a tool for the study of stochastic differential equations. In particular, in \cite{Ponosov2021} Ponosov proved a 
stochastic version of Brouwer fixed point theorem, which is stated as Proposition \ref{prop.2.4} in this paper and has played an 
essential role in the proofs of the two fixed point theorems --- Theorems 5.1 and 5.2 of \cite{Ponosov2022} that can be directly applied 
to the study of stochastic differential equations. In Section 5, we begin in proving that Ponosov's stochastic Brouwer fixed point theorem
\cite{Ponosov2021} is equivalent to a special case of our random Brouwer fixed point theorem by combining 
Ponosov's profound result \cite{Ponosov1987} and the theory of complete random inner product modules. In particular, 
we are able to obtain a somewhat surprising conclusion again by means of Ponosov's result \cite{Ponosov1987}: a $\mathcal{T}_{\varepsilon,\lambda}$-continuous $\sigma$-stable mapping from a $\mathcal{T}_{\varepsilon,\lambda}$-closed $L^{0}$-convex subset of $L^0(\mathcal{F},\mathbb{R}^d)$ to $L^0(\mathcal{F},\mathbb{R}^d)$ must be a.s. sequentially continuous (see Proposition \ref{prop.5.8}  of this paper), which shows that Proposition 3.1 of \cite{DKKS2013} is equivalent to either Proposition \ref{prop.2.3} or Proposition \ref{prop.2.4}.

\par The classical Borsuk theorem \cite{Borsuk1967} states that the unit sphere of an Euclidean space is not a retraction of the closed unit ball, 
and in turn is also equivalent to the classical Brouwer fixed point theorem. In \cite{DKKS2013}, a special random Borsuk theorem was proved, see Lemma 3.5 of \cite{DKKS2013}. In this paper we are able to prove a general random Borsuk theorem, and in particular we also prove that it is also equivalent to the random Brouwer fixed point theorem.

\par This paper is organized as follows.  In Section \ref{sec.2}, we provide some prerequisites and a survey of the random Brouwer fixed point theorem and its special cases currently in existence. In Section \ref{sec.3},  we establish  a key representation theorem of a proper $L^{0}$-labeling function by usual proper labeling functions, namely, Theorem \ref{thm.representation}. Based on Theorem \ref{thm.representation}, we prove Theorem \ref{thm.Sperner}, namely, a random Sperner lemma. As an application of Theorem \ref{thm.Sperner}, in Section \ref{sec.4}, we give a new complete proof of Theorem \ref{thm.Brouwer}.  
In Section \ref{sec.5}, we prove the equivalence of Proposition \ref{prop.2.3} and Proposition \ref{prop.2.4}. In Section \ref{sec.6}, we prove a general random Borsuk theorem --- Theorem \ref{thm.6.4} together with its equivalence with the random Brouwer fixed point theorem. Finally, in Section \ref{sec.7} we conclude this paper with commentaries on recent state of study of the famous Schauder conjecture.

\section{The random Brouwer fixed point theorem and a general random Sperner lemm. Preliminaries}\label{sec.2}

\subsection{The random Brouwer fixed point theorem and its special cases}\label{sec.2.1}

\par We begin by first introducing some basic notation and terminologies.

\par Throughout this paper, 
$(\Omega,\mathcal{F},P)$ denotes a given probability space,  $\mathbb{N}$ the set of positive integers and $\mathbb{K}$  the scalar 
field $\mathbb{R}$ of real numbers or $\mathbb{C}$ of complex numbers. We use $L^{0}(\mathcal{F},\mathbb{K})$ for the algebra of equivalence classes 
of  $\mathbb{K}$-valued random variables on $(\Omega,\mathcal{F},P)$, and $L^{0}(\mathcal{F},\mathbb{R}^{d})$ for the linear space of equivalence 
classes of random vectors from $(\Omega,\mathcal{F},P)$ to $\mathbb{R}^{d}$. As usual, $L^{0}(\mathcal{F},\mathbb{R})$ is briefly denoted 
by $L^{0}(\mathcal{F})$. $L^{0}(\mathcal{F},\mathbb{R}^{d})$ forms a free $L^{0}(\mathcal{F})$-module of rank $d$.  More generally, for any  separable metric space $M$,   $L^0(\mathcal{F},M)$ denotes the set of equivalence classes of $M$-valued $\mathcal{F}$-measurable functions on $(\Omega,\mathcal{F},P)$.

\par $\bar{L}^{0}(\mathcal{F})$ denotes the set of equivalence classes of extended real-valued random variables on $(\Omega,\mathcal{F},P)$.  
For any $\xi,\eta\in  \bar{L}^{0}(\mathcal{F})$, a partial order $\leq$ on $\bar{L}^{0}(\mathcal{F})$ is defined as follows: $\xi\leq \eta$ iff $\xi^0(\omega) \leq \eta^0(\omega)$ for almost all 
$\omega$ in $\Omega$ (briefly, $\xi^0(\omega) \leq \eta^0(\omega)$ a.s.), where $\xi^0$ and $\eta^{0}$ are arbitrarily chosen 
representatives of $\xi$ and $\eta$, respectively. Then $(\bar{L}^{0}(\mathcal{F}),\leq)$ forms a complete 
lattice (see \cite{DS1958} for details). 

\par As usual, $\xi< \eta$  means $\xi\leq \eta$ and $\xi\neq \eta$, whereas, for any $A\in \mathcal{F}$, $\xi< \eta$ on $A$ means 
$\xi^0(\omega)< \eta^0(\omega)$  for almost all $\omega\in A$, where $\xi^0$ and $\eta^0$ are again arbitrary representatives of $\xi$ and $\eta$, 
respectively. Furthermore, we denote $L_{+}^0(\mathcal{F}):=\{\xi\in L^0(\mathcal{F}):\xi\geq 0\}$ and 
$L_{++}^0(\mathcal{F}):=\{\xi\in L^0(\mathcal{F}):\xi>0~on~\Omega\}$. For any $A\in \mathcal{F}$,  we use $\tilde{I}_{A}$ for the 
equivalence class of $I_A$, where  $I_A$ is the characteristic function of $A$ (namely, $I_A(\omega)=1$ if $\omega\in A$ and $I_A(\omega)=0$ otherwise).

\par For any $\xi,\eta\in \bar{L}^{0}(\mathcal{F})$, we use $(\xi=\eta)$ for the equivalence class of $(\xi^{0}=\eta^{0})$,
where $(\xi^{0}=\eta^{0}):=\{\omega\in \Omega:\xi^{0}(\omega)=\eta^{0}(\omega)\}$, and $\xi^{0}$ and $\eta^{0}$ are arbitrarily chosen 
representatives of $\xi$ and $\eta$, respectively. Similarly, one can understand $(\xi<\eta)$ and $(\xi\leq\eta)$.

\par 
For any $x\in L^{0}(\mathcal{F}, \mathbb{R}^{d})$ and any $r\in L_{++}^{0}(\mathcal{F})$, the set $B(x,r):=\{ y \in L^{0}(\mathcal{F}, \mathbb{R}^{d}): \|x-y\|<r~\text{on}~\Omega\}$ 
is called the random open ball centered at $x$ with the random radius $r$, where the $L^{0}$-norm $\|\cdot\|$ is defined by $\|x\|=\sqrt{\sum_{i=1}^{d} x_{i}^2}$ for any 
$x=(x_{1},x_{2},\cdots,x_{d})\in L^{0}(\mathcal{F}, \mathbb{R}^{d})$. The collection $\{B(x,r):x\in L^{0}(\mathcal{F}, \mathbb{R}^{d}), r\in L_{++}^{0}(\mathcal{F})\}$ 
forms a base for some Hausdorff topology on $L^{0}(\mathcal{F}, \mathbb{R}^{d})$. Consequently,  $L^{0}(\mathcal{F}, \mathbb{R}^{d})$ can be  equipped 
with two different topologies: the topology of convergence in probability, denoted by $\mathcal{T}_{\varepsilon,\lambda}$, and the topology generated 
by random open balls, denoted by $\mathcal{T}_{c}$. $L^{0}(\mathcal{F},\mathbb{R}^{d})$ is called the random Euclidean space of rank $d$ when it is endowed with either $\mathcal{T}_{\varepsilon,\lambda}$ or $\mathcal{T}_{c}$.

\par  
A subset $G$ of $L^0(\mathcal{F},\mathbb{R}^{d})$ is said to be 
\begin{itemize}
	\item a.s. bounded if $\bigvee \{\|x\|:x\in G\}\in L_{+}^{0}(\mathcal{F})$.
	\item $L^{0}$-convex if $\lambda x+(1-\lambda) y\in G$ for any $x,y\in G$ and any $\lambda\in L_{+}^{0}(\mathcal{F})$ with $\lambda\leq 1$.
	\item $\sigma$-stable if $\sum_{n=1}^{\infty}\tilde{I}_{A_{n}}x_{n}\in G$ for any sequence 
	$\{x_{n},n\in \mathbb{N}\}$ in $G$ and any countable partition $\{A_n, n\in \mathbb{N} \}$ of $\Omega$ to $\mathcal{F}$ (namely, each 
	$A_n\in \mathcal{F}, A_i\cap A_j=\emptyset$ for any $i\neq j$, and $\bigcup_{n=1}^{\infty} A_n= \Omega$). Here, $\sum_{n=1}^{\infty}\tilde{I}_{A_{n}}x_{n}$ is well defined, namely,
	it allows a natural interpretation: arbitrarily choosing a representative $x_{n}^{0}$ of $x_{n}$ for each $n\in \mathbb{N}$, the function $\sum_{n=1}^{\infty}I_{A_{n}}x_{n}^{0}$ defined 
	by $(\sum_{n=1}^{\infty}I_{A_{n}}x_{n}^{0})(\omega)=x_{n}^{0}(\omega)$ when $\omega\in A_{n}$ for some $n\in \mathbb{N}$, is exactly a representative of $\sum_{n=1}^{\infty}\tilde{I}_{A_{n}}x_{n}$.
\end{itemize}
A mapping $f:G\rightarrow G$ is said to be 
\begin{itemize}
	\item $\sigma$-stable  if $G$ is $\sigma$-stable and $f(\sum_{n=1}^{\infty}\tilde{I}_{A_{n}}x_{n})=\sum_{n=1}^{\infty}\tilde{I}_{A_{n}}f(x_{n})$ 
	for any sequence $\{x_{n},n\in \mathbb{N}\}$ in $G$ and any countable partition $\{A_n, n\in \mathbb{N} \}$ of $\Omega$ to $\mathcal{F}$.
	\item $\mathcal{T}_{\varepsilon,\lambda}$-continuous if 
	$f$ is a continuous mapping from $(G,\mathcal{T}_{\varepsilon,\lambda})$ to $(G,\mathcal{T}_{\varepsilon,\lambda})$.
	\item $\mathcal{T}_{c}$-continuous if 
	$f$ is a continuous mapping from $(G,\mathcal{T}_{c})$ to $(G,\mathcal{T}_{c})$.
	\item a.s. sequentially continuous  if for any $x\in G$ and any sequence $\{x_{n},n\in \mathbb{N}\}$ in $G$ 
	such that $\{\|x_{n}-x\|,n\in \mathbb{N}\}$ converges a.s. to $0$, then we have $\{\|f(x_{n})-f(x)\|,n\in \mathbb{N}\}$ converges a.s. to $0$.
	\item random sequentially continuous if $G$ is $\sigma$-stable and for any $x\in G$ and any sequence $\{x_{n},n\in \mathbb{N}\}$ in $G$ convergent in $\mathcal{T}_{\varepsilon,\lambda}$ to $x$ there
	exists a random subsequence $\{x_{n_{k}},k\in \mathbb{N}\}$ of $\{x_{n},n\in \mathbb{N}\}$ such that $\{f(x_{n_{k}}),k\in \mathbb{N}\}$ converges in $\mathcal{T}_{\varepsilon,\lambda}$ to $f(x)$. Here, by
	$\{x_{n_{k}},k\in \mathbb{N}\}$ being a random subsequence of $\{x_{n},n\in \mathbb{N}\}$, we mean that $\{n_{k},k\in \mathbb{N}\}$ is a sequence of equivalence classes of positive-integer-valued 
	 random variables such that the following two conditions are satisfied:
	 \begin{itemize}
	 	\item [(1)] $n_{k}<n_{k+1}$ on $\Omega$ for any $k\in \mathbb{N}$;
	 	\item [(2)] $x_{n_{k}}=\sum_{l=1}^{\infty}I_{(n_{k}=l)}x_{l}$ for any $k\in \mathbb{N}$, where $n_{k}^{0}$ is an arbitrarily chosen representative of $n_{k}$, $(n_{k}=l)$ is the equivalence class of $A_{k,l}:=\{\omega\in\Omega:n_{k}^{0}(\omega)=l\}$ and $I_{(n_{k}=l)}=\tilde{I}_{A_{k,l}}$. 
	 \end{itemize} 
\end{itemize}

\par  Theorem \ref{thm.Brouwer} below is just the so called the random Brouwer fixed point theorem (Earlier, a sketch of proof, rather incomplete, was given in \cite[Lemma 4.6]{GWXYC2025}).

\begin{theorem}[The Random Brouwer Fixed Point Theorem]\label{thm.Brouwer}
	Let $G$ be a $\sigma$-stable, a.s. bounded, $\mathcal{T}_{c}$-closed and $L^0$-convex subset of $L^0(\mathcal{F},\mathbb{R}^d)$ 
	and  $f: G \rightarrow G$ be a $\sigma$-stable $\mathcal{T}_{c}$-continuous mapping. Then $f$ has a fixed point. 
\end{theorem}

\par Since $L^0(\mathcal{F},\mathbb{R}^{d})$ is a $\mathcal{T}_{\varepsilon,\lambda}$-complete random normed module, every $\mathcal{T}_{\varepsilon,\lambda}$-closed 
and $L^0$-convex subset of $L^0(\mathcal{F},\mathbb{R}^d)$ must be $\sigma$-stable, and Lemma 4.3 of \cite{GWXYC2025} shows that a $\sigma$-stable mapping $f$ is $\mathcal{T}_{c}$-continuous iff $f$ is random sequentially continuous.  It is also known from \cite{Guo2010} that a $\sigma$-stable subset of $L^0(\mathcal{F},\mathbb{R}^{d})$ is 
$\mathcal{T}_{\varepsilon,\lambda}$-closed iff it is $\mathcal{T}_{c}$-closed, so Proposition \ref{prop.2.2} below, which was first given as Lemma 4.6 of \cite{GWXYC2025}, 
is an equivalent formulation of Theorem \ref{thm.Brouwer}.

\begin{proposition}\label{prop.2.2}
	Let $G$ be an a.s. bounded, $\mathcal{T}_{\varepsilon,\lambda}$-closed and $L^0$-convex subset of $L^0(\mathcal{F},\mathbb{R}^d)$ 
	and  $f: G \rightarrow G$ be a $\sigma$-stable random sequentially continuous mapping. Then $f$ has a fixed point. 
\end{proposition}

\par It is obvious that a $\mathcal{T}_{\varepsilon,\lambda}$-continuous mapping defined on a $\sigma$-stable set is random sequentially continuous.  Proposition \ref{prop.2.3} below, which was first given in \cite[Proposition 4.2]{GZWW2021} is a special case of Proposition \ref{prop.2.2} (namely, Theorem \ref{thm.Brouwer}).

\begin{proposition}\label{prop.2.3}
	Let $G$ be an a.s. bounded, $\mathcal{T}_{\varepsilon,\lambda}$-closed and $L^0$-convex subset of $L^0(\mathcal{F},\mathbb{R}^d)$ 
	and  $f: G \rightarrow G$ be a $\sigma$-stable $\mathcal{T}_{\varepsilon,\lambda}$-continuous mapping. Then $f$ has a fixed point. 
\end{proposition}

\par 
Let $X$ and $Y$ be two separable metric spaces, it is well known that $L^{0}(\mathcal{F},X)$ and $L^{0}(\mathcal{F},Y)$ are metrizable spaces when they are endowed
with the topologies of convergence in probability, respectively. Further, let $G$ be a nonempty subset of $L^{0}(\mathcal{F},X)$, a mapping $h:G\rightarrow L^{0}(\mathcal{F},Y)$
is said to be local if $x^{0}(\omega)=y^{0}(\omega)$ a.s. on $A$ implies $h(x)^{0}(\omega)=h(y)^{0}(\omega)$ a.s. on $A$ for any $x$ and $y$ in $G$ and any $A\in \mathcal{F}$, where
$x^{0}$, $y^{0}$, $h(x)^{0}$ and $h(y)^{0}$ are any given representatives of $x$, $y$, $h(x)$ and $h(y)$, respectively. Finally, $h$ is said to be continuous in probability if 
$h$ is a continuous mapping from $G$ to $L^{0}(\mathcal{F},Y)$ when $G$ and $L^{0}(\mathcal{F},Y)$ are endowed with the topologies of convergence in probability. In fact, 
a mapping is continuous in probability iff it is $\mathcal{T}_{\varepsilon,\lambda}$-continuous.

\par As usual, $2^{\mathbb{R}^{d}}$ denotes the family of subsets of $\mathbb{R}^{d}$. Let $U:\Omega\rightarrow 2^{\mathbb{R}^{d}}$ be a set-valued mapping satisfying the following two conditions: (1) $U(\omega)$ is a nonempty bounded closed convex subset of $\mathbb{R}^{d}$
for each $\omega\in \Omega$; (2) $Gr(U):=\{(\omega,x)\in \Omega\times \mathbb{R}^{d}:x\in U(\omega)\}\in \mathcal{F}\otimes Bor(\mathbb{R}^{d})$, namely, $U$ has a measurable graph, where $Bor(\mathbb{R}^{d})$ stands
for the Borel $\sigma$-algebra of $\mathbb{R}^{d}$.

\par Further, for $U$ as above, we always denote by $G_{U}$ the set of equivalence classes of almost surely measurable selections of $U$, namely, 
$G_{U}:=\{x\in L^0(\mathcal{F},\mathbb{R}^d):x$ has a representative $x^{0}$ such that $x^{0}(\omega)\in U(\omega)~a.s.\}$.

\par  Then it is not difficult to verify that $G_{U}$ is an a.s. bounded $\mathcal{T}_{\varepsilon,\lambda}$-closed $L^{0}$-convex subset of $L^0(\mathcal{F},\mathbb{R}^d)$ (see Lemma \ref{lemm.5.3} of this paper), 
so $G_{U}$ is also $\sigma$-stable and it is easy to check that a mapping $h:G_{U}\rightarrow G_{U}$ is local iff $h$ is $\sigma$-stable (see Lemma \ref{lemm.5.1} of this paper).

\par Proposition \ref{prop.2.4} below, which was called the stochastic version of Brouwer fixed point theorem in \cite{Ponosov2021}, has played an essential role in the proofs of Theorems 5.1 and 5.2 of \cite{Ponosov2022}. It is obvious that  
Proposition \ref{prop.2.4} is a special case of Proposition \ref{prop.2.3}. The original proof of Proposition \ref{prop.2.4} was based on Ponosov's solution of the Nemytskii conjecture \cite{Ponosov1987} 
and a measurable selection theorem from \cite{Wagner1977}. In this paper, we can even prove that Proposition \ref{prop.2.4} is  equivalent to Proposition \ref{prop.2.3} by means of Ponosov's result
\cite{Ponosov1987}. In particular, we can also obtain a somewhat surprising conclusion again by means of Ponosov's result \cite{Ponosov1987}: a $\mathcal{T}_{\varepsilon,\lambda}$-continuous $\sigma$-stable mapping
from a $\mathcal{T}_{\varepsilon,\lambda}$-closed $L^{0}$-convex subset of $L^0(\mathcal{F},\mathbb{R}^d)$ to $L^0(\mathcal{F},\mathbb{R}^d)$ must be a.s. sequentially continuous, see Proposition \ref{prop.5.8}  of this paper.

\par Proposition \ref{prop.2.4} is a slightly more general formulation of Theorem 2.1 of \cite{Ponosov2021} (see also Lemma B.4 of \cite{Ponosov2022}) in that we have removed the assumption on completeness of $(\Omega,\mathcal{F},P)$ (otherwise, we first consider $(\Omega,\hat{\mathcal{F}},P)$, which is the completion of $(\Omega,\mathcal{F},P)$, and obtain a fixed point $x$ of $h$ as in \cite{Ponosov2021} and $x$ has a representative $x^{0}$ such that $x^{0}(\omega)\in U(\omega)$ a.s. and $x^{0}$ is $\hat{\mathcal{F}}$-measurable, whereas $x^{0}$ is equal a.s. to an $\mathcal{F}$-measurable function $\hat{x}^{0}$. Then the equivalence class $\hat{x}$ of $\hat{x}^{0}$ is exactly our desired fixed point).

\begin{proposition}\label{prop.2.4}
	Let $U:\Omega\rightarrow 2^{\mathbb{R}^{d}}$ be a set-valued function with a measurable graph such that 
	$U(\omega)$ is a nonempty bounded closed convex subset of $\mathbb{R}^{d}$ for each $\omega\in \Omega$. Then every local and continuous-in-probability
	operator from $G_{U}$ to itself has a fixed point.
\end{proposition}

\par A very special case of Proposition \ref{prop.2.3} first appeared as Proposition 3.1 in \cite{DKKS2013}, which states that
every $\sigma$-stable a.s. sequentially continuous  mapping from an a.s. bounded a.s. sequentially closed  $L^{0}$-convex subset of 
$L^0(\mathcal{F},\mathbb{R}^d)$ to itself has a fixed point, where a subset is said to be a.s. sequentially closed if it contains all a.s. limits of sequences in it. Proposition \ref{prop.2.3} is a generalization of 
Proposition 3.1 of \cite{DKKS2013} since a.s. sequentially closed set is exactly $\mathcal{T}_{\varepsilon,\lambda}$-closed and 
a.s. sequential continuity obviously  implies $\mathcal{T}_{\varepsilon,\lambda}$-continuity. It is interesting to notice that Proposition \ref{prop.5.8} implies the equivalence of Proposition 3.1 of  \cite{DKKS2013} with either of Proposition \ref{prop.2.3} or Proposition \ref{prop.2.4}.

\par We would also like to point out that Proposition \ref{prop.2.2} (namely, Theorem \ref{thm.Brouwer}) and Proposition \ref{prop.2.3} are both best possible, which is explained as follows. It is well known that $L^0(\mathcal{F},\mathbb{R}^d)$ is a metrizable linear topological space under $\mathcal{T}_{\varepsilon,\lambda}$ (namely, the topology of convergence in probability). There is, naturally, the notion of boundedness in the sense of the linear topology $\mathcal{T}_{\varepsilon,\lambda}$, denoted by the $\mathcal{T}_{\varepsilon,\lambda}$-boundedness. It is easy to see that the a.s. boundedness implies the $\mathcal{T}_{\varepsilon,\lambda}$-boundedness, and we have known from Proposition 16 of \cite{Guo2013} that these two kinds of boundedness coincide for a $\sigma$-stable subset. Likewise, the $L^{0}$-convexity implies the usual convexity, and we have also known from Lemma 2.8 of \cite{SGT2025} that these two kinds of convexity coincide for a  $\mathcal{T}_{\varepsilon,\lambda}$-closed and $\sigma$-stable subset. On the one hand, an a.s. bounded $\mathcal{T}_{\varepsilon,\lambda}$-closed $L^{0}$-convex subset $G$ of $L^{0}(\mathcal{F},\mathbb{R}^{d})$ is generally noncompact. If the hypothesis on the $\sigma$-stability of $f$ in Proposition \ref{prop.2.3} is removed, then Proposition \ref{prop.2.3} does not hold, which comes from a classical result due to Dobrowolski and Marciszewski \cite{DM1997} (see also \cite[Section 2.1]{GWXYC2025}). On the other hand, an example provided by Ponosov in \cite{Ponosov2021} shows that Proposition \ref{prop.2.4} (namely, Proposition \ref{prop.2.3}) does not hold if $G_{U}$ is replaced with a generic $\mathcal{T}_{\varepsilon,\lambda}$-closed, $\mathcal{T}_{\varepsilon,\lambda}$-bounded and convex set $G$ (namely, $G$ is not $L^{0}$-convex). Therefore, they are all best possible.

\subsection{A general $L^{0}$-simplicial subdivision and random Sperner lemma}\label{sec.2.2}

\par In this paper, we present a general theory of $L^{0}$-simplicial subdivisions of $L^0$-simplexes and give a thorough treatment of a random Sperner lemma by utilizing the inherent connection between $L^0$-simplexes and classical simplexes. So we can smoothly present a good definition of an $L^{0}$-simplicial subdivision of an $L^0$-simplex (see Definition \ref{defn.$L^{0}$-simplicial subdivision} below) and establish a general random Sperner lemma (see Theorem \ref{thm.Sperner} below). To survey these main results, we first introduce some notation and prerequisites.

\par  In the sequel of this paper, a left module over the algebra $L^{0}(\mathcal{F},\mathbb{K})$ is simply said to be an $L^{0}$-module.

\par Since $L^{0}$-labeling functions take values in $L^{0}(\mathcal{F},\{1,2,\cdots,n\})$ (see Definition \ref{defn.labeling function} below), to remove any possible
ambiguities by strictly distinguishing a measurable set and its equivalence class, we employ the language from measure algebras. For this, we introduce the convention
from measure algebras as follows.

\par  For any $A, B\in \mathcal{F}$,  we will always use the corresponding lowercase letters $a$ and $b$  for the equivalence 
classes $[A]$ and $[B]$ (two elements $C$ and $D$ in $\mathcal{F}$ are said to be equivalent if $P[(C\setminus D)\cup(D\setminus C)]=0$), 
respectively. Define $1=[\Omega]$, $0=[\emptyset]$, $a\wedge b=[A\cap B]$, $a\vee b=[A\cup B] $ and $a^{c}=[A^{c}]$, 
where $A^c$ denotes the complement of $A$. Let $B_{\mathcal{F}}=\{a=[A]:A\in \mathcal{F}\}$, then $(B_{\mathcal{F}},\wedge,\vee,^{c},0,1)$ 
is a complete Boolean algebra, namely, a complete complemented distributive lattice (see \cite{Kop1989} for details). Specifically, $B_{\mathcal{F}}$ is called the measure algebra associated with $(\Omega,\mathcal{F},P)$.

\par As usual, for any $a,b\in B_{\mathcal{F}}$, $a>b$ means $a\geq b$ and $a\neq b$. 
A subset $\{a_{i}:i\in I\}$ of $B_{\mathcal{F}}$ is called a partition of unity if $\bigvee_{i\in I} a_{i} =1$ and $a_{i} \wedge a_{j} = 0$ 
for any $i,j\in I$ with $i\neq j$. The collection of all such partitions is denoted by $p(1)$. For any $\{a_{j}:j\in J\}\in p(1)$, 
it is clear that the set $\{j\in J : a_{j}>0\}$ is at most countable. Moreover, $\{a_{n},n\in \mathbb{N}\}\in p(1)$ iff  there exists 
a countable partition $\{A_n, n\in \mathbb{N} \}$ of $\Omega$ to $\mathcal{F}$ such that $a_{n}=[A_{n}]$ for each $n\in \mathbb{N}$.

\par From now on, for any $a=[A]\in B_{\mathcal{F}}$, we also use $I_{a}$ 
to denote $\tilde{I}_{A}$.  For a subset $H$ of a complete lattice (e.g., $B_{\mathcal{F}}$), we use $\bigvee H$ and 
$\bigwedge H$ to denote the supremum and infimum of $H$, respectively.

\par Recall that an  $L^{0}$-module $E$ is said to be regular if $E$ has the following property: for any given 
two elements $x$ and $y$ in $E$, if there exists $\{a_{n},n\in \mathbb{N}\}\in p(1)$ such that $I_{a_{n}}x=I_{a_{n}}y$ 
for each $n\in \mathbb{N}$, then $x=y$.

\par In the remainder of this paper, all the $L^{0}$-modules under consideration are assumed to be regular. 
This restriction is not excessive, since all random normed modules (for example, $L^{0}(\mathcal{F},\mathbb{R}^{d})$) are regular.

\par Since our aim is to develop random functional analysis and its applications to dynamic Nash equilibrium theory and nonsmooth differential geometry on metric 
measure spaces (see \cite{BPS2023,CLPV2025,CGP2025,Gigli2018,GLP2025,LP2019,LPV2024} for recent advances), random normed modules will play an active role and they are all more general $L^{0}$-modules than $L^{0}(\mathcal{F},\mathbb{R}^{d})$, and thus we state the results of this section in terms of $L^{0}$-modules for possible convenience in future applications.

\par In the special $L^{0}$-module $L^{0}(\mathcal{F},\mathbb{R}^{d})$, $\sum_{n=1}^{\infty}I_{a_{n}}x_{n}$ is always well defined for any sequence $\{x_{n},n\in \mathbb{N}\}$
 and any $\{a_{n},n\in \mathbb{N}\}\in p(1)$. However, for a generic $L^{0}$-module $E$, $\sum_{n=1}^{\infty}I_{a_{n}}x_{n}$ may be meaningless for an arbitrary sequence $\{x_{n},n\in \mathbb{N}\}$ in $E$ and arbitrary $\{a_{n},n\in \mathbb{N}\}\in p(1)$. For this, we need the notion of a $\sigma$-stable set introduced in \cite{Guo2010} (see also (4) of Definition \ref{defn.$L^{0}$-simplex} below) in order to speak of a countable concatenation such as $\sum_{n=1}^{\infty}I_{a_{n}}x_{n}$ in a meaningful way.

\begin{definition}[\cite{Guo2010,GWXYC2025}]\label{defn.$L^{0}$-simplex}
	Let $E$ be an $L^{0}$-module and $G$ be a nonempty subset of $E$.
	\begin{enumerate}
		\item [(1)] $G$ is said to be $L^{0}$-convex if $\lambda x+(1-\lambda) y\in G$ for any $x,y\in G$ and any $\lambda\in L_{+}^{0}(\mathcal{F})$ with $\lambda\leq 1$.
		Generally, for any subset $G$, 
		 $Conv_{L^0} (G):= \{ \sum^{n}_{i=1} \xi_{i} x_{i}: n\in \mathbb{N}, \{x_{1},x_{2},\cdots, x_{n}\}\subseteq G~\text{and}~\{\xi_{1},\xi_{2},\cdots,\xi_{n}\}\subseteq L^0_+(\mathcal{F})~\text{such that}~ \sum^{n}_{i=1} \xi_{i}=1\}$ is called the $L^0$-convex hull of $G$. Clearly, $G$ is $L^0$-convex if $Conv_{L^0}(G)=G$.
		
		\item [(2)] A finite subset $\{x_{1},x_{2},\cdots, x_{n}\}$ of $G$ 
		is said to be $L^0$-linearly independent if, for any $\{\xi_{1},\xi_{2},\cdots,\xi_{n}\} \subseteq L^0(\mathcal{F}, \mathbb{K})$ satisfying $\sum^{n}_{i=1} \xi_i x_i=0$, it always holds that 
		$\xi_i=0$ for each $i=1\sim n$. Further, $G$ is said to be $L^0$-linearly independent if any finite nonempty subset of $G$ is $L^0$-linearly independent, 
		in which case the $L^0$-submodule generated by $G$ is said to be $L^0$-free.
		
		\item [(3)] A finite subset $\{x_{1},x_{2},\cdots, x_{n} \}$ of $E$ is said to be $L^0$-affinely independent if 
		either $n=1$ or if $n>1$ and $\{x_i-x_1: i=2\sim n \}$ is $L^0$-linearly independent, 
		in which case $Conv_{L^0} (\{x_{1}, x_{2},\cdots, x_{n} \})$ is called an $(n-1)$-dimensional $L^0$-simplex with 
		vertices $x_{1}, x_{2},\cdots, x_n$ and $Conv_{L^{0}}(\{x_{i}:i\in I\})$ is called an $(|I|-1)$-dimensional $L^{0}$-face of 
		$Conv_{L^0} (\{x_{1}, x_{2},\cdots, x_{n} \})$ for any nonempty subset $I$ of $\{1,2,\cdots, n \}$, where $|I|$ denotes the cardinality of $I$. 
		
		\item [(4)] $G$ is said to be finitely stable if $I_{a}x+I_{a^{c}}y\in G$ for any $x,y\in G$ and any $a\in B_{\mathcal{F}}$. $G$ is said to be $\sigma$-stable
		(or to have the countable concatenation property in terms of \cite{Guo2010}) if, for each sequence $\{x_{n}, n\in \mathbb{N} \}$ in $G$ and each 
		$\{a_{n},n\in \mathbb{N}\}\in p(1)$, there exists some $x\in G$ such that $I_{a_{n}}x=I_{a_{n}}x_n$ 
		for each $n\in \mathbb{N}$. Such an $x$ is unique since $E$ is assumed to be regular, usually denoted by $ \sum_{n=1}^{\infty}I_{a_{n}}x_n$, 
		called the countable concatenation of $\{x_n,n\in\mathbb{N}\}$ along $\{a_n,n\in \mathbb{N}\}$. If $G$ is
		$\sigma$-stable and  $H$ is a nonempty subset of $G$, then $\sigma(H):=\{ \sum_{n=1}^{\infty}I_{a_{n}}h_n: \{h_n,n\in \mathbb{N}\}$
		is a sequence in $H$ and $\{a_{n},n\in \mathbb{N}\}\in p(1)\}$ is called the $\sigma$-stable hull of $H$.
	\end{enumerate}
\end{definition}

\par 
For any $(n-1)$-dimensional $L^0$-simplex $S= Conv_{L^0} (\{x_{1},x_{2},\cdots, x_{n}\})$ in an $L^{0}$-module $E$, 
as shown by Guo et al. \cite{GWXYC2025}, $S$ is $\sigma$-stable and
each element $x\in S$ can be uniquely represented as $x= \sum^n_{i=1} \xi_i x_i$, where
$\xi_i\in L^0_+(\mathcal{F})$ for any $i\in \{1,2,\cdots,n\}$ such that $ \sum^n_{i=1} \xi_i =1$. 
Moreover, since  $E$ is also a linear space over $\mathbb{K}$ and the $L^{0}$-affine independence of $\{x_{1},x_{2},\cdots, x_{n}\}$  
implies their affine independence,  then $S'= Conv(\{x_1, x_2,\cdots, x_n \})$ is a classical  
$(n-1)$-dimensional simplex in the linear space $E$, where $Conv(G)$ denotes the usual convex hull of a nonempty subset $G$ in a linear space.
In this case, we call $S'$  the associated simplex of $S$. Throughout this paper, whenever we refer to an $L^{0}$-simplex $S$, we will always denote its associated simplex by $S'$.

\par 
Recall from \cite{Bor1985} that for any $(n-1)$-dimensional simplex $S'=Conv(\{x_{1},x_{2},$ $\cdots,x_{n}\})$ in a linear space, a finite family  $\mathscr{S}'$  of $(n-1)$-dimensional 
simplexes is called a simplicial subdivision of $S'$ if
\begin{itemize}
	\item [(1)] $S'=\bigcup \mathscr{S}'$;
	\item [(2)] for any $S'_{1},S'_{2}\in \mathscr{S}'$, $S'_{1}\cap S'_{2}$ is either 
	empty or a common face of $S'_{1}$ and $S'_{2}$.
\end{itemize}

\par As a random generalization, we introduce the notion of an $L^{0}$-simplicial subdivision of an $L^0$-simplex 
in Definition \ref{defn.$L^{0}$-simplicial subdivision} below.

\begin{definition}\label{defn.$L^{0}$-simplicial subdivision}
	Let $S=Conv_{L^0}(\{x_{1}, x_{2},\cdots, x_{n}\})$ be an $(n-1)$-dimensional $L^0$-simplex in an $L^{0}$-module $E$.
	A finite family $\mathscr{S}$  of $(n-1)$-dimensional $L^0$-simplexes is called an $L^{0}$-simplicial subdivision of $S$
	if 
	\begin{enumerate} 
		\item [(1)] $S=\sigma\{\bigcup \mathscr{S}\}$;
		\item [(2)] for any $S_{1},S_{2}\in \mathscr{S}$, $S_{1}\cap S_{2}$ is either empty or a common $L^{0}$-face of $S_{1}$ and $S_{2}$;
		\item [(3)] $\mathscr{S}'=\{G':G\in \mathscr{S}\}$ is a simplicial subdivision of $S'$.
	\end{enumerate}
\end{definition}

\par Let $\mathscr{S}$ and $\mathscr{S}'$ be the same as in Definition \ref{defn.$L^{0}$-simplicial subdivision}, then we call $\mathscr{S}'$ the associated simplicial subdivision of $\mathscr{S}$. Throughout this paper, whenever we refer to an $L^{0}$-simplicial subdivision $\mathscr{S}$, we will always denote its associated simplicial subdivision by $\mathscr{S}'$. $L^{0}$-simplexes are much more complicated objects than usual simplexes, causing considerable challenges in our initial study of them. These challenges motivate us to introduce  Definition \ref{defn.$L^{0}$-simplicial subdivision} in connection with the classical simplicial subdivision  
so that we can make full use of the theory of classical simplicial subdivisions to give a  thorough treatment of $L^{0}$-simplicial subdivisions of $L^{0}$-simplexes.

\par 
The notion of $L^0$-extreme points was first introduced and studied by Guo et al. in \cite{GWXYC2025}. 
For two nonempty subsets $G$ and $H$ of an $L^{0}$-module (resp., a linear space)
such that $H \subseteq G$. $H$ is called an $L^0$-extreme subset (resp., an extreme subset) of $G$ if both 
$x$ and $y$ belong to $H$ whenever $x$ and $y$ belong to $G$ and there exists $\lambda\in L^0(\mathcal{F}, (0,1))$ 
(resp., $\lambda\in (0,1)$) such that $\lambda x+(1-\lambda)y\in H$. Further, if $H=\{h\}$ is a singleton and 
an $L^0$-extreme subset (resp., an extreme subset) of $G$, then $h$ is called an $L^0$-extreme point (resp., an extreme point) 
of $G$. We will always denote by $ext_{L^0} (G)$ (resp., $ext(G)$) the set of $L^0$-extreme points (resp., extreme points) 
of $G$. For any family $\mathscr{G}$ of nonempty subsets of an $L^{0}$-module (resp., a linear space), the set of $L^0$-extreme points (resp., extreme points) 
of $\mathscr{G}$ is defined by $ext_{L^{0}}(\mathscr{G}):=\sigma[\bigcup_{G\in \mathscr{G}}ext_{L^{0}}(G)]$ 
(resp., $ext(\mathscr{G}):=\bigcup_{G\in \mathscr{G}}ext(G)$).

\par For any $(n-1)$-dimensional simplex $S'=Conv(\{x_{1},x_{2},\cdots,x_{n}\})$ in a linear space and any 
simplicial subdivision $\mathscr{S}'$ of $S'$, a mapping $\psi: ext(\mathscr{S}')\rightarrow \{1,2,\cdots,n\}$
is called a proper labeling function of $\mathscr{S}'$ (see \cite{Bor1985}) if  
$$\psi(y)\in \chi(y), \forall y\in ext(\mathscr{S}'),$$
where $\chi(y):=\{i:\alpha_{i}>0\}$ for any $y=\sum_{i=1}^{n}\alpha_{i}x_{i}$.  In fact, the set of all proper labeling functions is exactly $\Pi_{y\in ext(\mathscr{S}')}\chi(y)$.

\par A simplex $G'=Conv(\{y_{1},y_{2},\cdots,y_{n}\})$  with $y_{j}\in ext(\mathscr{S}')$ for each $j\in \{1,2,\cdots,n\}$ is said to be completely labeled by $\psi$ if $\psi$ is a 
proper labeling function of $\mathscr{S}'$ and  
$$\{\psi(y_{j}):j=1\sim n\}=\{1,2,\cdots,n\}.$$

\par Let $E_{1}$ and $E_{2}$ be two $L^{0}$-modules, and let $G_{1}\subseteq E_{1}$ and $G_{2}\subseteq E_{2}$ be two
$\sigma$-stable sets. A mapping $f:G_{1}\rightarrow G_{2}$ is said to be $\sigma$-stable if  
$$f(\sum^{\infty}_{n=1} I_{a_{n}} x_{n}) = \sum^{\infty}_{n=1} I_{a_{n}} f(x_{n})$$
for any  $\{a_{n}, n\in \mathbb{N}\}\in p(1)$ and any sequence $\{x_{n}, n\in \mathbb{N} \}$ in $G_{1}$.

\par In \cite[Definition 2.1]{DKKS2013}, Drapeau et al. introduced the notion of a proper $L^{0}$-labeling function 
for the $L^{0}$-barycentric subdivision of an $L^{0}$-simplex in $L^{0}(\mathcal{F},\mathbb{R}^{d})$.  
Definition~\ref{defn.labeling function} below is a generalization of \cite[Definition 2.1]{DKKS2013}.
But we adopt a strengthened notion of a completely labeled $L^{0}$-simplex since such a notion is both easy to understand and
such a completely labeled $L^{0}$-simplex indeed exists (see Theorem \ref{thm.Sperner}). In fact, according to Definition~\ref{defn.labeling function},
$\phi(y_{j})$ is an equivalence class of a constant valued random variable for each $j\in \{1,2,\cdots,n\}$ for a completely labeled $L^{0}$-simplex $G=Conv_{L^{0}}(\{y_{1},y_{2},\cdots,y_{n}\})$,
and hence our notion of a completely labeled $L^{0}$-simplex is completely similar to the notion of a completely labeled simplex.

\begin{definition}\label{defn.labeling function}
	Let $S=Conv_{L^{0}}(\{x_{1},x_{2},\cdots,x_{n}\})$ be an $(n-1)$-dimensional $L^{0}$-simplex in 
	an $L^{0}$-module $E$, $\mathscr{S}$ an $L^{0}$-simplicial subdivision 
	of $S$ and $\phi:ext_{L^{0}}(\mathscr{S})\rightarrow L^{0}(\mathcal{F},\{1,2,\cdots,n\})$. The mapping  
	$\phi$ is called an $L^{0}$-labeling function of $\mathscr{S}$ if $\phi$ is $\sigma$-stable. An $L^{0}$-labeling 
	function $\phi$ of $\mathscr{S}$ is said to be proper if 
	$$(\phi(y)=i)\wedge (\lambda_{i}=0)=0~\text{for~any}~y\in ext_{L^{0}}(\mathscr{S})~\text{and~any}~i\in \{1,2,\cdots,n\},$$
	where $y=\sum_{i=1}^{n}\lambda_{i}x_{i}$. An $L^{0}$-simplex $G=Conv_{L^{0}}(\{y_{1},y_{2},\cdots,y_{n}\})$ with 
	$y_{j}\in ext_{L^{0}}(\mathscr{S})$ for each $j\in \{1,2,\cdots,n\}$ is said to be completely labeled by $\phi$ if 
	$\phi$ is a proper $L^{0}$-labeling function of $\mathscr{S}$ and  
    $$\{\phi(y_{j}):j=1\sim n\}=\{1,2,\cdots,n\}.$$
\end{definition}

\par  The classical Sperner  lemma, originally established by Sperner in \cite{Sperner1928} was formulated for simplexes in $\mathbb{R}^{n}$. Let $S'=Conv(\{x_{1},x_{2},\cdots,x_{n}\})$
be an $(n-1)$-dimensional simplex in a linear space $L$ and $L^{n}=Span(\{x_{1},x_{2},\cdots,x_{n}\})$ be the linear subspace of $L$ generated by $\{x_{1},x_{2},\cdots,x_{n}\}$. Then $L^{n}$
is at most $n$-dimensional and the classical Sperner lemma is only related to $L^{n}$ and is independent of $L$, so one can readily see the following version of Sperner 
lemma for simplexes in a general linear space.

\begin{lemma}\label{lemm.Sperner}
	Let $S'=Conv (\{x_{1},x_{2},\cdots,x_{n}\})$ be an $(n-1)$-dimensional simplex in 
	a linear space, $\mathscr{S}'$  a simplicial subdivision 
	of $S'$ and $\psi$ a proper labeling function of $\mathscr{S}'$. 
	Then there are an odd number of completely labeled simplexes in $\mathscr{S}'$ .
\end{lemma}

\par For any given sequence $\{G_{n},n\in \mathbb{N}\}$ of nonempty subsets of a $\sigma$-stable set $G$ of an $L^{0}$-module and any given
$\{a_{n},n\in \mathbb{N}\}\in p(1)$, 
$$\sum_{n=1}^{\infty}I_{a_{n}}G_{n}:=\{\sum_{n=1}^{\infty}I_{a_{n}}x_{n}:x_{n}\in G_{n}~\text{for~each}~n\in \mathbb{N}\}$$
is called the  countable concatenation of $\{G_{n},n\in\mathbb{N}\}$ along $\{a_n,n\in \mathbb{N}\}$.

\par Theorem \ref{thm.Sperner} below can be aptly called a random Sperner lemma.

\begin{theorem}[Random Sperner Lemma]\label{thm.Sperner}
	Let $S=Conv_{L^{0}}(\{x_{1},x_{2},\cdots,x_{n}\})$ be an $(n-1)$-dimensional $L^{0}$-simplex in 
	an $L^{0}$-module $E$, $\mathscr{S}$  an $L^{0}$-simplicial subdivision 
	of $S$ and $\phi$ a proper $L^{0}$-labeling function of $\mathscr{S}$. 
	Then there exist a finite family $\{S_{m}:=Conv_{L^{0}}(\{y_{1}^{m},y_{2}^{m},\cdots,y_{n}^{m}\}),m=1\sim M\}$ 
	of $(n-1)$-dimensional $L^{0}$-simplexes in $\mathscr{S}$ and $\{a_{m},m=1\sim M\}\in p(1)$ such that 
	$$\sum_{m=1}^{M}I_{a_{m}}S_{m}=Conv_{L^{0}}(\{\sum_{m=1}^{M}I_{a_{m}}y_{1}^{m}, \sum_{m=1}^{M}I_{a_{m}}y_{2}^{m},\cdots, \sum_{m=1}^{M}I_{a_{m}}y_{n}^{m}\})$$ 
	is an $(n-1)$-dimensional $L^{0}$-simplex completely labeled by $\phi$ and satisfies that 
	$$\phi(\sum_{m=1}^{M}I_{a_{m}}y_{j}^{m})=j~\text{for~any}~j\in \{1,2,\cdots,n\}.$$
\end{theorem}

\par 
In Theorem \ref{thm.Sperner},  $\sum_{m=1}^{M}I_{a_{m}}S_{m}$ is an $(n-1)$-dimensional $L^{0}$-simplex contained in $S$ such that 
each vertex of $\sum_{m=1}^{M}I_{a_{m}}S_{m}$ belongs to $ext_{L^{0}}(\mathscr{S})$ since both $S$ and $ext_{L^{0}}(\mathscr{S})$ are $\sigma$-stable.
Besides, $\sum_{m=1}^{M}I_{a_{m}}S_{m}$ is not necessarily in $\mathscr{S}$ but merely a finite concatenation of the finite subfamily $\{S_{m}:m=1\sim M\}$ of $\mathscr{S}$, which is a significant difference between Theorem \ref{thm.Sperner} and the classical Sperner lemma ({namely, Lemma \ref{lemm.Sperner}).

\section{Proof of Theorem \ref{thm.Sperner}: Random Sperner Lemma} \label{sec.3}

\par We first give the following two examples --- Examples \ref{ex.$L^{0}$-barycentric subdivision} and \ref{ex.1}. The first one shows
the $L^{0}$-barycentric subdivision introduced in \cite[Definition 1.10]{DKKS2013}  is a special case of our Definition \ref{defn.$L^{0}$-simplicial subdivision}.
The second one shows that our Definition \ref{defn.$L^{0}$-simplicial subdivision} surpasses the $L^{0}$-barycentric subdivision.

\begin{example}\label{ex.$L^{0}$-barycentric subdivision}
	Let $S=Conv_{L^{0}}(\{x_{1},x_{2},\cdots,x_{n}\})$ be an $(n-1)$-dimensional $L^{0}$-simplex in 
	an $L^{0}$-module $E$ and $G(n)$ be the group of permutations of $\{1,2,\cdots,n\}$. For any $\pi\in G(n)$, define
	\begin{align}
		&y_{k}^{\pi}=\frac{1}{k}\sum_{i=1}^{k}x_{\pi(i)}, \forall k\in \{1,2,\cdots,n\},\nonumber\\
		&S_{\pi}=Conv_{L^{0}}(y_{1}^{\pi},y_{2}^{\pi},\cdots,y_{n}^{\pi}).\nonumber
	\end{align}
	Then, by applying the same method as in \cite[Lemma 1.11]{DKKS2013}, one can verify that $\{S_{\pi}:\pi\in G(n)\}$ is 
	an $L^{0}$-simplicial subdivision of $S$, which is called the $L^{0}$-barycentric subdivision of $S$. 
\end{example}

\begin{example}\label{ex.1}
	Let $S=Conv_{L^{0}}(\{x_{1},x_{2},x_{3}\})$ be a $2$-dimensional $L^{0}$-simplex in an $L^{0}$-module $E$. Let $y_{1,2}=\frac{1}{2}(x_{1}+x_{2})$, 
	$y_{1,3}=\frac{1}{2}(x_{1}+x_{3})$ and  $y_{2,3}=\frac{1}{2}(x_{2}+x_{3})$, it is easy to check that  
	$$\mathscr{S}=\left\{
	\begin{array}{l}
		\begin{array}{l}
			Conv_{L^{0}}(\{x_{1},y_{1,2},y_{1,3}\}),Conv_{L^{0}}(\{x_{2},y_{1,2},y_{2,3}\}),\\
			Conv_{L^{0}}(\{x_{3},y_{1,3},y_{2,3}\}),Conv_{L^{0}}(\{y_{1,2},y_{1,3},y_{2,3}\})
		\end{array}
	\end{array}
	\right\}$$
	is a family of $2$-dimensional $L^{0}$-simplexes. 
	
	\par (1) On the one hand, for any $x=\lambda_{1}x_{1}+\lambda_{2}x_{2}+\lambda_{3}x_{3}$ in $S$, where 
		$\lambda_{1},\lambda_{2},\lambda_{3}\in L_{+}^{0}(\mathcal{F}) $ and $\lambda_{1}+\lambda_{2}+\lambda_{3}=1$, 
		let
		$$
		\begin{array}{ll}
			b_{1}=(\lambda_{1}\geq \frac{1}{2}),&\quad b_{2}=(\lambda_{2}\geq \frac{1}{2}),\\
			b_{3}=(\lambda_{3}\geq \frac{1}{2}),&\quad b_{4}=(\lambda_{1}< \frac{1}{2})\wedge (\lambda_{2}< \frac{1}{2})\wedge (\lambda_{3}< \frac{1}{2})
		\end{array}
		$$
		and further let 
		$$
		\begin{array}{ll}
			a_{1}=b_{1},&\quad a_{2}=b_{2}\wedge b_{1}^{c},\\
			a_{3}=b_{3}\wedge (b_{1}\vee b_{2})^{c},&\quad a_{4}=b_{4},
		\end{array}
		$$
		then $\{a_{1},a_{2},a_{3},a_{4}\}\in p(1)$ and we have
		\begin{align}
			x&=\lambda_{1}x_{1}+\lambda_{2}x_{2}+\lambda_{3}x_{3}\nonumber\\
			&=I_{a_{1}}[(\lambda_{1}-\lambda_{2}-\lambda_{3})x_{1}+2\lambda_{2}y_{1,2}+2\lambda_{3}y_{1,3}]+\nonumber\\
			&\quad I_{a_{2}}[(\lambda_{2}-\lambda_{1}-\lambda_{3})x_{2}+2\lambda_{1}y_{1,2}+2\lambda_{3}y_{2,3}]+\nonumber\\
			&\quad I_{a_{3}}[(\lambda_{3}-\lambda_{1}-\lambda_{2})x_{3}+2\lambda_{1}y_{1,3}+2\lambda_{2}y_{2,3}]+\nonumber\\
			&\quad I_{a_{4}}[(\lambda_{1}+\lambda_{2}-\lambda_{3})y_{1,2}+(\lambda_{1}+\lambda_{3}-\lambda_{2})y_{1,3}+(\lambda_{2}+\lambda_{3}-\lambda_{1})y_{2,3}]\nonumber\\
			& \in I_{a_{1}} Conv_{L^{0}}(\{x_{1},y_{1,2},y_{1,3}\})+I_{a_{2}} Conv_{L^{0}}(\{x_{2},y_{1,2},y_{2,3}\})+\nonumber\\
			&\quad I_{a_{3}} Conv_{L^{0}}(\{x_{3},y_{1,3},y_{2,3}\})+I_{a_{4}} Conv_{L^{0}}(\{y_{1,2},y_{1,3},y_{2,3}\})\nonumber\\
			&\subseteq \sigma\{\bigcup \mathscr{S}\},\nonumber
		\end{align}
		implying $S\subseteq\sigma\{\bigcup \mathscr{S}\}$. On the other hand, since $S$ is both $\sigma$-stable and $L^{0}$-convex, 
		it follows that $\sigma\{\bigcup \mathscr{S}\} \subseteq S$. Therefore, $S=\sigma\{\bigcup \mathscr{S}\}$.
 
    \par (2) For any $z\in Conv_{L^{0}}(\{x_{1},y_{1,2},y_{1,3}\})\cap Conv_{L^{0}}(\{y_{1,2},y_{1,3},y_{2,3}\})$, let 
    $$z=\xi_{1}x_{1}+\xi_{2}y_{1,2}+\xi_{3}y_{1,3}=\eta_{1}y_{2,3}+\eta_{2}y_{1,3}+\eta_{3}y_{1,2},$$
    where $\{\xi_{1},\xi_{2},\xi_{3}\}\subseteq L_{+}^{0}(\mathcal{F})$, $\{\eta_{1},\eta_{2},\eta_{3}\}\subseteq L_{+}^{0}(\mathcal{F})$ and 
    $$\xi_{1}+\xi_{2}+\xi_{3}=\eta_{1}+\eta_{2}+\eta_{3}=1,$$
    we have 
    $$z=(\xi_{1}+\frac{\xi_{2}}{2}+\frac{\xi_{3}}{2})x_{1}+\frac{\xi_{2}}{2}x_{2}+\frac{\xi_{3}}{2}x_{3}=\frac{\eta_{2}+\eta_{3}}{2}x_{1}+ \frac{\eta_{1}+\eta_{3}}{2}x_{2}+\frac{\eta_{1}+\eta_{2}}{2}x_{3}.$$
    Since $\{x_{1},x_{2},x_{3}\}$ is $L^{0}$-affinely independent,  
    $$
    \begin{cases}
    	\begin{aligned}
    		\xi_{1}+\frac{\xi_{2}}{2}+\frac{\xi_{3}}{2}&=\frac{\eta_{2}+\eta_{3}}{2}\\
    		\frac{\xi_{2}}{2}&=\frac{\eta_{1}+\eta_{3}}{2}\\
    		\frac{\xi_{3}}{2}&=\frac{\eta_{1}+\eta_{2}}{2}
    	\end{aligned}
    \end{cases},
    $$
    then  
    $$
    \begin{cases}
    	\begin{aligned}
    		\xi_{1}&=\eta_{1}=0\\
    		\xi_{2}&=\eta_{3}\\
    		\xi_{3}&=\eta_{2}
    	\end{aligned}
    \end{cases},
    $$
    which implies that $y\in Conv_{L^{0}}(\{y_{1,2},y_{1,3}\})$, and hence 
    $$Conv_{L^{0}}(\{x_{1},y_{1,2},y_{1,3}\})\cap Conv_{L^{0}}(\{y_{1,2},y_{1,3},y_{2,3}\})=Conv_{L^{0}}(\{y_{1,2},y_{1,3}\}).$$
    Similarly, one can prove that 
    $$Conv_{L^{0}}(\{x_{2},y_{1,2},y_{2,3}\})\cap Conv_{L^{0}}(\{y_{1,2},y_{1,3},y_{2,3}\})=Conv_{L^{0}}(\{y_{1,2},y_{2,3}\}),$$
    $$Conv_{L^{0}}(\{x_{3},y_{1,3},y_{2,3}\})\cap Conv_{L^{0}}(\{y_{1,2},y_{1,3},y_{2,3}\})=Conv_{L^{0}}(\{y_{1,3},y_{2,3}\}),$$
    $$Conv_{L^{0}}(\{x_{1},y_{1,2},y_{1,3}\})\cap Conv_{L^{0}}(\{x_{2},y_{1,2},y_{2,3}\})=\{y_{1,2}\},$$
    $$Conv_{L^{0}}(\{x_{1},y_{1,2},y_{1,3}\})\cap Conv_{L^{0}}(\{x_{3},y_{1,3},y_{2,3}\})=\{y_{1,3}\}$$
    and 
    $$Conv_{L^{0}}(\{x_{2},y_{1,2},y_{2,3}\})\cap Conv_{L^{0}}(\{x_{3},y_{1,3},y_{2,3}\})=\{y_{2,3}\}.$$

    Therefore,  $\mathscr{S}$ satisfies (2) of Definition \ref{defn.$L^{0}$-simplicial subdivision}.
    
    \par (3) Similar to the argument presented in (2), one can readily verify that
	$$\mathscr{S}'=\left\{
	\begin{array}{l}
		\begin{array}{l}
			Conv(\{x_{1},y_{1,2},y_{1,3}\}),Conv(\{x_{2},y_{1,2},y_{2,3}\}),\\
			Conv(\{x_{3},y_{1,3},y_{2,3}\}),Conv(\{y_{1,2},y_{1,3},y_{2,3}\})
		\end{array}
	\end{array}
	\right\}$$
	is a simplicial subdivision of $S'$.
	
	\par To sum up, $\mathscr{S}$ is an $L^{0}$-simplicial subdivision of $S$.
\end{example}

\par To prove Theorem \ref{thm.Sperner}, we need the following Lemma \ref{lemm.$L^{0}$-simplicial subdivision}, Lemma \ref{lemm.$L^{0}$-simplicial subdivision connection} 
and Theorem \ref{thm.representation}. Although part (2) of Lemma \ref{lemm.$L^{0}$-simplicial subdivision} is almost clear since each $S_{i}$
is $\sigma$-stable, we still give its proof here for the reader's convenience.

\begin{lemma}\label{lemm.$L^{0}$-simplicial subdivision}
	Let $S=Conv_{L^0} (\{x_{1}, x_{2},\cdots, x_{n} \})$ be an $(n-1)$-dimensional $L^0$-simplex in an $L^{0}$-module $E$. Then we have the following properties:
	\begin{enumerate} 
		\item [(1)] For any sequence $\{S^{m}=Conv_{L^{0}}(\{x_{1}^{m},x_{2}^{m},\cdots,x_{n}^{m}\}),m\in \mathbb{N}\}$ of 
		$(n-1)$-dimensional $L^{0}$-simplexes contained in $S$ and any $\{a_{m},m\in \mathbb{N}\}\in p(1)$, $\sum_{m=1}^{\infty}I_{a_{m}}S^{m}$
		is an $(n-1)$-dimensional $L^{0}$-simplex contained in $S$ with vertexes $ \sum_{m=1}^{\infty}I_{a_{m}}x_{1}^{m}, \sum_{m=1}^{\infty}I_{a_{m}}x_{2}^{m},\cdots,$  $\sum_{m=1}^{\infty}I_{a_{m}}x_{n}^{m}$.
		\item [(2)] For any $L^{0}$-simplicial subdivision $\mathscr{S}=\{S_{i},i=1\sim I\}$ of $S$, 
		$$S=\bigcup\{ \sum_{i=1}^{I}I_{b_{i}}S_{i}:\{b_{i},i=1\sim I\}\in p(1)\}.$$
	\end{enumerate}
\end{lemma}
\begin{proof}
	(1) First, we prove that 
	$$\sum_{m=1}^{\infty}I_{a_{m}}S^{m}=Conv_{L^{0}}(\{\sum_{m=1}^{\infty}I_{a_{m}}x_{1}^{m}, \sum_{m=1}^{\infty}I_{a_{m}}x_{2}^{m},\cdots, \sum_{m=1}^{\infty}I_{a_{m}}x_{n}^{m}\}).$$
	For any $\{\lambda_{i},i=1\sim n\}\subseteq L_{+}^{0}(\mathcal{F})$ with $\sum_{i=1}^{n}\lambda_{i}=1$, we have 
	$$\sum_{i=1}^{n}\lambda_{i}(\sum_{m=1}^{\infty}I_{a_{m}}x_{i}^{m})=\sum_{m=1}^{\infty}I_{a_{m}}(\sum_{i=1}^{n}\lambda_{i}x_{i}^{m})\in \sum_{m=1}^{\infty}I_{a_{m}}S^{m},$$
	which implies that 
	$$Conv_{L^{0}}(\{\sum_{m=1}^{\infty}I_{a_{m}}x_{1}^{m}, \sum_{m=1}^{\infty}I_{a_{m}}x_{2}^{m},\cdots, \sum_{m=1}^{\infty}I_{a_{m}}x_{n}^{m}\})\subseteq\sum_{m=1}^{\infty}I_{a_{m}}S^{m}.$$
	On the other hand, let $\sum_{i=1}^{n}\lambda_{i}^{m}x_{i}^{m}\in S^{m}$ for each $m\in \mathbb{N}$, where  
	$\{\lambda_{i}^{m},i=1\sim n\}\subseteq L_{+}^{0}(\mathcal{F})$ and $\sum_{i=1}^{n}\lambda_{i}^{m}=1$, 
	since $\{\sum_{m=1}^{\infty}I_{a_{m}}\lambda_{i}^{m},i=1\sim n\}\subseteq L_{+}^{0}(\mathcal{F})$  and 
	$$\sum_{i=1}^{n}(\sum_{m=1}^{\infty}I_{a_{m}}\lambda_{i}^{m})=\sum_{m=1}^{\infty}I_{a_{m}}(\sum_{i=1}^{n}\lambda_{i}^{m})=1,$$
    we have
	\begin{align}
		\sum_{m=1}^{\infty}I_{a_{m}}(\sum_{i=1}^{n}\lambda_{i}^{m}x_{i}^{m})&=\sum_{i=1}^{n}(\sum_{m=1}^{\infty}I_{a_{m}}\lambda_{i}^{m}x_{i}^{m})\nonumber\\
		&=\sum_{i=1}^{n}(\sum_{m=1}^{\infty}I_{a_{m}}\lambda_{i}^{m})(\sum_{m=1}^{\infty}I_{a_{m}}x_{i}^{m})\nonumber\\
		&\in Conv_{L^{0}}(\{\sum_{m=1}^{\infty}I_{a_{m}}x_{1}^{m}, \sum_{m=1}^{\infty}I_{a_{m}}x_{2}^{m},\cdots, \sum_{m=1}^{\infty}I_{a_{m}}x_{n}^{m}\}),\nonumber
	\end{align}
	giving
	$$\sum_{m=1}^{\infty}I_{a_{m}}S^{m}\subseteq Conv_{L^{0}}(\{\sum_{m=1}^{\infty}I_{a_{m}}x_{1}^{m}, \sum_{m=1}^{\infty}I_{a_{m}}x_{2}^{m},\cdots, \sum_{m=1}^{\infty}I_{a_{m}}x_{n}^{m}\}).$$

	\par Second, we prove that $\{\sum_{m=1}^{\infty}I_{a_{m}}x_{1}^{m}, \sum_{m=1}^{\infty}I_{a_{m}}x_{2}^{m},\cdots, \sum_{m=1}^{\infty}I_{a_{m}}x_{n}^{m}\}$ is $L^0$-affinely independent. 
	Let $\{\mu_{i},i=1\sim n\}\subseteq L^0(\mathcal{F},\mathbb{K})$ with $ \sum_{i=1}^{n}\mu_{i}=0$ and 
	$$\sum_{i=1}^{n}\mu_{i}(\sum_{m=1}^{\infty}I_{a_{m}}x_{i}^{m})=0.$$
	For any $m\in \mathbb{N}$, since 
	$$\sum_{i=1}^{n}(I_{a_{m}}\mu_{i})=I_{a_{m}}(\sum_{i=1}^{n}\mu_{i})=0,$$
	$$\sum_{i=1}^{n}(I_{a_{m}}\mu_{i})x_{i}^{m}=I_{a_{m}}[\sum_{i=1}^{n}\mu_{i}(\sum_{m=1}^{\infty}I_{a_{m}}x_{i}^{m})]=0$$
	and $S^{m}$ is an $(n-1)$-dimensional $L^{0}$-simplex, we have
	$$I_{a_{m}}\mu_{i}=0, \forall i\in \{1,2,\cdots,n\}.$$
	It follows from the arbitrariness of $m$ that 
	$$\mu_{i}=0, \forall i\in \{1,2,\cdots,n\}.$$

	\par (2) Since $S$ is $\sigma$-stable, for any $\{b_{i},i=1\sim I\}\in p(1)$, we have 
	$$\sum_{i=1}^{I}I_{b_{i}}S_{i}\subseteq \sum_{i=1}^{I}I_{b_{i}}(\bigcup_{i\in I}S_{i})\subseteq \sum_{i=1}^{I}I_{b_{i}}S=S,$$ 
	implying
	$$\bigcup\{\sum_{i=1}^{I}I_{b_{i}}S_{i}:\{b_{i},i=1\sim I\}\in p(1)\}\subseteq S.$$
	
	\par On the other hand, for any $x\in S$, since $S=\sigma(\bigcup_{i=1}^{I}S_{i})$, there exist $\{a_{n},n\in \mathbb{N}\}\in p(1)$ and a sequence 
	$\{x_{n},n\in \mathbb{N}\}$ in $\bigcup_{i=1}^{I}S_{i}$ such that $x=\sum_{n=1}^{\infty}I_{a_{n}}x_{n}$.  
	Let 
	$$N_{i}=\{n\in \mathbb{N}:x_{n}\in S_{i}\},\forall i\in\{1,2,\cdots,I\}.$$
	For each $i\in\{1,2,\cdots,I\}$, since $S_{i}$ is $\sigma$-stable, arbitrarily fixing a $z_{i}\in S_{i}$, then there exists $y_{i}\in S_{i}$ such that 
	$I_{a_{n}}x_{n}=I_{a_{n}}y_{i}$ for any $n\in N_{i}$ and $I_{(\bigvee_{n\in N_{i}}a_{n})^{c}}z_{i}=I_{(\bigvee_{n\in N_{i}}a_{n})^{c}}y_{i}$.
	Let 
	$$b_{i}=\bigvee_{n\in N_{i}}a_{n}, \forall i\in\{1,2,\cdots,I\}.$$ 
	It is clear that $\{b_{i},i=1\sim I\}\in p(1)$ and 
	\begin{align}
		x&=\sum_{n=1}^{\infty}I_{a_{n}}x_{n}\nonumber\\
		&=\sum_{i=1}^{I}I_{b_{i}}y_{i}\nonumber\\
		&\in \sum_{i=1}^{I}I_{b_{i}} S_{i}\nonumber\\
		&\subseteq \bigcup\{\sum_{i=1}^{I}I_{b_{i}}S_{i}:\{b_{i},i=1\sim I\}\in p(1)\}.\nonumber
	\end{align}
\end{proof}

\par As a direct application of Lemma 5.4 of \cite{GWXYC2025}, Lemma \ref{lemm.$L^{0}$-simplicial subdivision connection} below gives an interesting connection between the set of $L^{0}$-extreme points of an $L^{0}$-simplicial subdivision and the set of extreme points of its associated simplicial subdivision.

\begin{lemma}\label{lemm.$L^{0}$-simplicial subdivision connection}
	Let $E$ and $S$ be the same as in Lemma \ref{lemm.$L^{0}$-simplicial subdivision}, and $\mathscr{S}=\{S_{i},i=1\sim I\}$ be 
	an $L^{0}$-simplicial subdivision  of $S$. Then we have the following relations:
	\begin{enumerate}
		\item [(1)] $ext_{L^0}(S)=\sigma(ext(S'))$.
		\item [(2)] $ext_{L^{0}}(\mathscr{S})=\sigma(ext(\mathscr{S}'))$.
	\end{enumerate}
\end{lemma}
\begin{proof}
	(1) By Lemma 5.4 of \cite{GWXYC2025}, we have 
	$$ext_{L^0}(S)=\sigma(\{x_{1},x_{2},\cdots, x_{n}\}),$$
	which, combined with the fact $ext(S')=\{x_{1}, x_{2},\cdots, x_{n}\}$, 
	implies that 
	$$ext_{L^0}(S)=\sigma(ext(S')).$$
	
	\par (2) Since $ext_{L^{0}}(S_{i})= \sigma(ext(S'_{i}))$ for each $i\in \{1,2,\cdots,I\}$, we have 
	\begin{align}
		ext_{L^{0}}(\mathscr{S})&=\sigma(\bigcup_{i=1}^{I} ext_{L^{0}}(S_{i}))\nonumber\\
		&=\sigma(\bigcup_{i=1}^{I} \sigma(ext(S'_{i})))\nonumber\\
		&=\sigma(\bigcup_{i=1}^{I} ext(S'_{i}))\nonumber\\
		&=\sigma(ext(\mathscr{S}')).\nonumber
	\end{align}
\end{proof}

\par For any $\xi=\sum_{i=1}^{n}I_{a_{i}} i\in L^{0}(\mathcal{F},\{1,2,\cdots,n\})$,  in such an expression of $\xi$, from now on we always assume that $a_{i}=(\xi=i)$ for each $i\in \{1,2,\cdots,n\}$ (then it is clear that $\{a_{i},i=1\sim n\}\in p(1)$). We will also always use $\mathcal{N}_{\xi}$ for the set $\{i:a_{i}>0\}$.

\par For any $\{a_{n},n\in \mathbb{N}\},~\{b_{m},m\in \mathbb{N}\}\in p(1)$, it is easy to check that $\{a_{n}\wedge b_{m},n,m\in \mathbb{N}\} \in p(1)$, 
which is referred to as the refinement of $\{a_{n},n\in \mathbb{N}\}$ and $\{b_{m},m\in \mathbb{N}\}$. The refinement of any finite collection of partitions can be defined inductively.

\par Theorem \ref{thm.representation} below establishes a key connection between the set of proper $L^{0}$-labeling functions of $\mathscr{S}$
and the set of proper labeling functions of $\mathscr{S'}$. In particular, part (3) of Theorem \ref{thm.representation} below gives the representation theorem of a proper $L^{0}$-labeling function by usual proper labeling functions, which considerably simplifies the proof of Theorem \ref{thm.Sperner}.

\begin{theorem}\label{thm.representation}
	Let $S=Conv_{L^{0}}(\{x_{1},x_{2},\cdots,x_{n}\})$ be an $(n-1)$-dimensional $L^{0}$-simplex in an $L^{0}$-module $E$ and  
	$\mathscr{S}$ be an $L^{0}$-simplicial subdivision of $S$. Then we have the following statements:
	\begin{enumerate} 
		\item [(1)] For any proper $L^{0}$-labeling function $\phi$ of $\mathscr{S}$, 
		$$\mathcal{N}_{\phi(y)}\subseteq  \chi(y), \forall y\in ext(\mathscr{S}').$$
		
		\item [(2)] For any proper $L^{0}$-labeling function $\phi$ of $\mathscr{S}$ and any $y\in ext(\mathscr{S}')$, there exists 
		a family $\{\psi_{i}:i\in \mathcal{N}_{\phi(y)}\}$ of proper labeling functions of $\mathscr{S'}$ with $\psi_{i}(y)=i$ for each $i\in \mathcal{N}_{\phi(y)}$, namely,
		$$\phi(y)=\sum_{i\in \mathcal{N}_{\phi(y)}}I_{(\phi(y)=i)}\psi_{i}(y).$$
		
		\item [(3)] Denote $ext(\mathscr{S}')=\{y_{k}:k=1\sim K\}$ and let $\phi$ be a proper $L^{0}$-labeling function of $\mathscr{S}$. Then 
		there exist a finite family $\{\psi_{m},m=1\sim M\}$ of proper labeling functions of $\mathscr{S}'$ and $\{a_{m},m=1\sim M\}\in p(1)$ such that 
		$$\phi(y_{k})=\sum_{m=1}^{M}I_{a_{m}}\psi_{m}(y_{k})~\text{for~any}~k\in \{1,2,\cdots,K\}.$$
		More generally, since $\phi$ is $\sigma$-stable, we also have
		$$\phi(y)=\sum_{k=1}^{K}I_{b_{k}} \sum_{m=1}^{M}I_{a_{m}}\psi_{m}(y_{k})$$
		for any $y=\sum_{k=1}^{K}I_{b_{k}}y_{k}\in ext_{L^{0}}(\mathscr{S})$ with $\{b_{k},k=1\sim K\}\in p(1)$.
	\end{enumerate}	
\end{theorem}
\begin{proof}
	(1) For any given $y\in ext(\mathscr{S}')\subseteq S'$, there exists $\{\alpha_{i},i=1\sim n\}\subseteq [0,1]$ with $\sum_{i=1}^{n}\alpha_{i}=1$ such that 
	$$y=\sum_{i=1}^{n} \alpha_{i}x_{i}\in S.$$
	For any $i\in \mathcal{N}_{\phi(y)}$,  if $\alpha_{i}=0$, then $(\alpha_{i}=0)=1$, which, combined with  
	$$(\phi(y)=i)\wedge (\alpha_{i}=0)=0,$$
	implies that $(\phi(y)=i)=0$. This contradicts the fact that $i\in \mathcal{N}_{\phi(y)}$.  Hence, $\alpha_{i}>0$ and this implies 
	that $i\in \chi(y)$. Thus, 
	$$\mathcal{N}_{\phi(y)}\subseteq  \chi(y).$$

	\par (2) For each $i\in \mathcal{N}_{\phi(y)}$, since $\mathcal{N}_{\phi(y)}\subseteq  \chi(y)$, it is clear that 
	there exists a proper labeling function $\psi_{i}$ of $\mathscr{S'}$ such that $\psi_{i}(y)=i$, namely,
	$$\phi(y)=\sum_{i\in \mathcal{N}_{\phi(y)}}I_{(\phi(y)=i)}\psi_{i}(y).$$
	
	\par (3) For each $k\in \{1,2,\cdots,K\}$, let $y_{k}=\sum_{i=1}^{n}\lambda_{i}^{k}x_{i}$, 
	where $\{\lambda_{i}^{k},i=1\sim n\}\subseteq [0,1]$ are such that $\sum_{i=1}^{n}\lambda_{i}^{k}=1$.
	Since $\phi$ maps $ext_{L^{0}}(\mathscr{S})$ into $L^{0}(\mathcal{F},\{1,2,\cdots,n\})$, 
	for each $k\in \{1,2,\cdots,K\}$, there exists $\{a_{i}^{k},i=1\sim n\}\in p(1)$ such that  
	$$\phi(y_{k})=\sum_{i=1}^{n}I_{a_{i}^{k}}i.$$
	Now, we consider the refinement 
	$$\{ a_{i_{1}}^{1}\wedge a_{i_{2}}^{2}\wedge \cdots\wedge a_{i_{K}}^{K}:  (i_{1}, i_{2},\cdots, i_{K})\in \{1,2,\cdots,n\}^{K}\}$$
	of $\{a_{i}^{1},i=1\sim n\},\cdots,\{a_{i}^{K-1},i=1\sim n\}$ and $\{a_{i}^{K},i=1\sim n\}$, where $\{1,2,\cdots,n\}^{K}$ is the $K$-th Cartesian power set of $\{1,2,\cdots,n\}$. 
	Further, let 
	$$I=\{(i_{1}, i_{2},\cdots, i_{K})\in \{1,2,\cdots,n\}^{K}: a_{i_{1}}^{1}\wedge a_{i_{2}}^{2}\wedge \cdots\wedge a_{i_{K}}^{K}>0\}.$$
	Then  
	$$\{ a_{i_{1}}^{1}\wedge a_{i_{2}}^{2}\wedge \cdots\wedge a_{i_{K}}^{K}: (i_{1}, i_{2},\cdots, i_{K})\in I\}\in p(1)$$
	and 
	$$\mathcal{N}_{\phi(y_{k})}=\{P_{k}(i_{1}, i_{2},\cdots, i_{K}):(i_{1}, i_{2},\cdots, i_{K})\in I\}, \forall k\in \{1,2,\cdots,K\},$$ 
	where $P_{k}(\cdot)$ denotes the projection of $\{1,2,\cdots,n\}^{K}$ onto the $k$-th coordinate.
	
	\par 
	For each $k\in \{1,2,\cdots,K\}$, let 
	$$\mathcal{N}_{\phi(y_{k})}(\text{namely,}~ P_{k}(I))=\{p_{k,l}:1\leq l\leq M_{k}\},$$
	where $M_{k}$ denotes the cardinality of $P_{k}(I)$. Further, for each $l\in \{1,2,\cdots,M_{k}\}$, let 
	$$I_{k,l}=\{(i_{1}, i_{2},\cdots, i_{K})\in I:P_{k}(i_{1}, i_{2},\cdots, i_{K})=p_{k,l}\},$$
	that is to say, $I_{k,l}=P_{k}^{-1}(p_{k,l})$. Then it is clear that $I=\bigcup_{l=1}^{M_{k}}I_{k,l}$ and $\{I_{k,l}:1\leq l\leq M_{k}\}$ is a pairwise disjoint family. 
	Now, by (2) as has been proved, there exists a family $\{\varphi_{k,l}:1\leq l\leq M_{k}\}$ of proper labeling functions of $\mathscr{S}'$ such that $\varphi_{k,l}(y_{k})=p_{k,l}$ for each $l\in \{1,2,\cdots M_{k}\}$, namely,
	$$\phi(y_{k})=\sum_{l=1}^{M_{k}}I_{(\phi(y_{k})=p_{k,l})}p_{k,l}=\sum_{l=1}^{M_{k}}I_{(\phi(y_{k})=p_{k,l})}\varphi_{k,l}(y_{k}).$$

	\par    
	For each  $k\in \{1,2,\cdots,K\}$ and each $l\in \{1,2,\cdots,M_{k}\}$, we define a proper labeling function $\psi_{(i_{1}, i_{2},\cdots, i_{K})}^{k}$ of $\mathscr{S}'$ for each $(i_{1}, i_{2},\cdots, i_{K})\in I_{k,l} $ by $\psi_{(i_{1}, i_{2},\cdots, i_{K})}^{k}=\varphi_{k,l}$. Then one can easily see that 
	$$\psi_{(i_{1}, i_{2},\cdots, i_{K})}^{k}(y_{k})=\varphi_{k,l}(y_{k})=p_{k,l}.$$
	To sum up, since $I=\bigcup_{l=1}^{M_{k}}I_{k,l}$, we can eventually obtain a family 
	$$\{\psi_{(i_{1}, i_{2},\cdots, i_{K})}^{k}:k\in \{1,2,\cdots,K\}~\text{and}~(i_{1}, i_{2},\cdots, i_{K})\in I\}$$
	of proper labeling functions of $\mathscr{S}'$ such that each $\psi_{(i_{1}, i_{2},\cdots, i_{K})}^{k}$ in this family satisfies that 
	$$\psi^{k}_{(i_{1}, i_{2},\cdots, i_{K})}(y_{k})=P_{k}(i_{1}, i_{2},\cdots, i_{K}).$$

	\par 
	Since, for each $k\in \{1,2,\cdots,K\}$ and each $l\in \{1,2,\cdots,M_{k}\}$,  
	$$(\phi(y_{k})=p_{k,l})=\bigvee_{(i_{1}, i_{2},\cdots, i_{K})\in I_{k,l}} a_{i_{1}}^{1}\wedge a_{i_{2}}^{2}\wedge \cdots\wedge a_{i_{K}}^{K},$$
	then 
	$$I_{(\phi(y_{k})=p_{k,l})}=\sum_{(i_{1}, i_{2},\cdots, i_{K})\in I_{k,l}}I_{ a_{i_{1}}^{1}\wedge a_{i_{2}}^{2}\wedge \cdots\wedge a_{i_{K}}^{K}},$$
	and hence 
	\begin{align}
		\phi(y_{k})&=\sum_{l=1}^{M_{k}}I_{(\phi(y_{k})=p_{k,l})}p_{k,l}\nonumber\\
		&=\sum_{l=1}^{M_{k}}\sum_{(i_{1}, i_{2},\cdots, i_{K})\in I_{k,l}}I_{ a_{i_{1}}^{1}\wedge a_{i_{2}}^{2}\wedge \cdots\wedge a_{i_{K}}^{K}}p_{k,l}\nonumber\\
		&=\sum_{l=1}^{M_{k}}\sum_{(i_{1}, i_{2},\cdots, i_{K})\in I_{k,l}}I_{ a_{i_{1}}^{1}\wedge a_{i_{2}}^{2}\wedge \cdots\wedge a_{i_{K}}^{K}}\psi_{(i_{1}, i_{2},\cdots, i_{K})}^{k}(y_{k})\nonumber\\
		&=\sum_{(i_{1}, i_{2},\cdots, i_{K})\in I}I_{ a_{i_{1}}^{1}\wedge a_{i_{2}}^{2}\wedge \cdots\wedge a_{i_{K}}^{K}}\psi_{(i_{1}, i_{2},\cdots, i_{K})}^{k}(y_{k}).\nonumber	
	\end{align}

	\par 
	Now, for each  $(i_{1}, i_{2},\cdots, i_{K})\in I$, we define a mapping $\psi_{(i_{1}, i_{2},\cdots, i_{K})}:ext(\mathscr{S'})\rightarrow \{1,2,\cdots,n\}$ by
	$$\psi_{(i_{1}, i_{2},\cdots, i_{K})}(y_{k})=\psi^{k}_{(i_{1}, i_{2},\cdots, i_{K})}(y_{k}), \forall k\in \{1,2,\cdots,K\}. $$
	Then  $\psi_{(i_{1}, i_{2},\cdots, i_{K})}$ is a proper labeling function of $\mathscr{S}'$.  It follows that
	$$\phi(y_{k})=\sum_{(i_{1}, i_{2},\cdots, i_{K})\in I}I_{ a_{i_{1}}^{1}\wedge a_{i_{2}}^{2}\wedge \cdots\wedge a_{i_{K}}^{K}} \psi_{(i_{1}, i_{2},\cdots, i_{K})}(y_{k}), \forall k\in \{1,2,\cdots,K\}.$$
	Furthermore, for any $y=\sum_{k=1}^{K}I_{b_{k}}y_{k}\in ext_{L^{0}}(\mathscr{S})$ with $\{b_{k},k=1\sim K\}\in p(1)$, we have 
	\begin{align}
		\phi(\sum_{k=1}^{K}I_{b_{k}}y_{k})&=\sum_{k=1}^{K}I_{b_{k}}\phi(y_{k})\nonumber\\
		&=\sum_{k=1}^{K}I_{b_{k}}(\sum_{(i_{1}, i_{2},\cdots, i_{K})\in I}I_{ a_{i_{1}}^{1}\wedge a_{i_{2}}^{2}\wedge \cdots\wedge a_{i_{K}}^{K}} \psi_{(i_{1}, i_{2},\cdots, i_{K})}(y_{k})).\nonumber
	\end{align}
\end{proof}

\par Lemma \ref{lemm.$L^{0}$-simplicial subdivision}, Lemma \ref{lemm.$L^{0}$-simplicial subdivision connection}  and Theorem \ref{thm.representation} have shown that $L^{0}$-simplexes, $L^{0}$-simplicial subdivisions and  proper $L^{0}$-labeling functions are closely connected to their classical counterparts. These results motivate us to make full use of Lemma \ref{lemm.Sperner} to complete the proof of Theorem \ref{thm.Sperner}.

\par We are now in a position to prove Theorem \ref{thm.Sperner}.

\begin{proof}[Proof of Theorem \ref{thm.Sperner}]
	Let $ext(\mathscr{S}')=\{y_{k}:k=1\sim K\}$ and $y_{k}=\sum_{i=1}^{n}\lambda_{i}^{k}x_{i}$ 
	with $\{\lambda_{i}^{k},i=1\sim n\}\subseteq [0,1]$ and $ \sum_{i=1}^{n}\lambda_{i}^{k}=1$ for any $k\in \{1,2,\cdots,K\}$.  By (3) of Theorem \ref{thm.representation},  
	there exist a finite family $\{\psi_{m},m=1\sim M\}$ of proper 
	labeling functions of $\mathscr{S}'$ and $\{a_{m},m=1\sim M\}\in p(1)$ such that 
	$$\phi(y_{k})=\sum_{m=1}^{M}I_{a_{m}}\psi_{m}(y_{k}),~\forall k\in \{1,2,\dots, K\}.$$

	\par 
	For each $\psi_{m}$, by  Lemma \ref{lemm.Sperner}, there exists $S'_{m}\in \mathscr{S}'$ such that $S'_{m}$ is completely labeled by $\psi_{m}$. 
	Let $S'_{m}=Conv(\{y_{1}^{m},\cdots,y_{n}^{m}\})$, where $y_{j}^{m}\in ext(\mathscr{S}')$ and  
	\begin{equation}\label{eq.thm.Sperner}
		\psi_{m}(y_{j}^{m})=j, \forall j\in \{1,2,\cdots,n\}.
	\end{equation}
	Further, let $S_{m}=Conv_{L^{0}}(\{y_{1}^{m},\cdots,y_{n}^{m}\})$ for each $m\in \{1,2,\cdots,M\}$. Then by part (1) of Lemma \ref{lemm.$L^{0}$-simplicial subdivision},
	$$\sum_{m=1}^{M}I_{a_{m}}S_{m}=Conv_{L^{0}}(\{\sum_{m=1}^{M}I_{a_{m}}y_{1}^{m},\sum_{m=1}^{M}I_{a_{m}}y_{2}^{m},\cdots,\sum_{m=1}^{M}I_{a_{m}}y_{n}^{m}\})$$
	is an $(n-1)$-dimensional $L^{0}$-simplex contained in $S$. Furthermore, by (\ref{eq.thm.Sperner}), we have 
	\begin{align}
		\phi(\sum_{m=1}^{M}I_{a_{m}}y_{j}^{m})&=\sum_{m=1}^{M}I_{a_{m}}\phi(y_{j}^{m})\nonumber\\
		&=\sum_{m=1}^{M}I_{a_{m}} (\sum_{l=1}^{M}I_{a_{l}}\psi_{l}(y_{j}^{m}))\nonumber\\
		&=\sum_{m=1}^{M}I_{a_{m}}\psi_{m}(y_{j}^{m})\nonumber\\
		&=\sum_{m=1}^{M}I_{a_{m}}j\nonumber\\
		&=j\nonumber
	\end{align}
	for any  $j\in \{1,2,\cdots,n\}$. This implies that $\sum_{m=1}^{M}I_{a_{m}}S_{m}$ is completely labeled by $\phi$.
\end{proof}

\section{Proof of Theorem \ref{thm.Brouwer}: The Random Brouwer Fixed Point Theorem}\label{sec.4}

\par 
Although Theorem \ref{thm.Brouwer} is stated in $L^{0}(\mathcal{F},\mathbb{R}^{d})$, several key concepts 
that will be used in its proof originate from the theory of random normed modules. For clarity and convenience, 
we begin this section by recalling some basic notions from the theory of random normed modules, and we also restate several concepts 
previously introduced in Section~\ref{sec.2} within the framework of random normed modules.

\par Definition \ref{defn. RN module} here is adopted from \cite{Guo1992,Guo1993} by following the traditional nomenclature of random metric spaces and random normed spaces (see \cite[Chapters 9 and
15]{SS2005}).

\begin{definition}[\cite{Guo1992,Guo1993}]\label{defn. RN module}
	An ordered pair $(E, \|\cdot\|)$ is called a random normed module (briefly, an $RN$ module) 
	over $\mathbb{K}$ with base $(\Omega, \mathcal{F}, P)$ if $E$ is an $L^{0}$-module
	and $\|\cdot\|$ is a mapping from $E$ to $L^0_+(\mathcal{F})$ such that the following conditions are satisfied:
	\begin{enumerate} 
		\item [(1)] $\|\xi x\|= |\xi| \|x\|$ for any $\xi \in L^0(\mathcal{F}, \mathbb{K})$ and any $x\in E$, 
		where $\xi x $ stands for the module multiplication of $x$ by $\xi$ and $|\xi|$ is the equivalence 
		class of $|\xi^0|$ for an arbitrarily chosen representative $\xi^0$ of $\xi$ (of course, $|\xi^0|$ 
		is the function defined by $|\xi^0|(\omega)= |\xi^0(\omega)|$ for any $\omega\in \Omega$);
		\item [(2)] $\|x+y\|\leq \|x\|+ \|y\|$ for any $x$ and $y$ in $E$;
		\item [(3)] $\|x\|=0$ implies $x=\theta$ (the null element of $E$). 
	\end{enumerate}
	As usual, $\|\cdot\|$ is called the $L^0$-norm on $E$.
\end{definition}

\par  It should be mentioned that the notion of an $L^0$-normed $L^0$-module, which is
equivalent to that of an $RN$ module, was independently introduced by Gigli in \cite{Gigli2018}
for the study of nonsmooth differential geometry on metric measure spaces, where the
$L^0$-norm was called the pointwise norm.

\par 
Clearly, when $(\Omega, \mathcal{F}, P)$ is trivial, namely, $\mathcal{F}=\{\Omega, \emptyset\}$, then an 
$RN$ module with base $(\Omega, \mathcal{F}, P)$ just reduces to an ordinary normed space.

\par Let $(E, \|\cdot\|)$ be an $RN$ module $(E, \|\cdot\|)$ over $\mathbb{K}$ with base $(\Omega, \mathcal{F}, P)$. 
For any real numbers  $\varepsilon>0$ and $0<\lambda<1$,  let 
$$U_{\theta}(\varepsilon,\lambda)=\{x\in E:P\{\omega\in \Omega:\|x\|(\omega)<\varepsilon\}>1-\lambda\}.$$
Then $\mathcal{U}_{\theta}=\{U_{\theta}(\varepsilon,\lambda): \varepsilon>0,~0<\lambda<1\}$ forms a local base of some 
metrizable linear topology for $E$, called the $(\varepsilon,\lambda)$-topology induced by $\|\cdot\|$, denoted by 
$\mathcal{T}_{\varepsilon,\lambda}$. The idea of introducing the $(\varepsilon,\lambda)$-topology for
$RN$ modules originates from Schweizer and Sklar's work on the theory of probabilistic metric spaces \cite{SS2005}. 
It is clear that a sequence $\{x_{n},n\in \mathbb{N}\}$ in $E$ converges in $\mathcal{T}_{\varepsilon,\lambda}$ to $x\in E$ iff 
$\{\|x_{n}-x\|,n\in \mathbb{N}\}$ converges in probability $P$ to $\theta$.  On the other hand, 
for any $x\in E$ and any $r\in L_{++}^{0}(\mathcal{F})$, let $$B(x,r)=\{y\in E:\|x-y\|<r~\text{on}~\Omega\}.$$
Then $\{B(x,r):x\in E, r\in L_{++}^{0}(\mathcal{F})\}$ forms a base for some Hausdorff topology on $E$, 
which amounts to the locally $L^{0}$-convex topology \cite{FKV2009} induced by $\|\cdot\|$, denoted by $\mathcal{T}_{c}$. 
It is well known from \cite{Guo2010} that $(E,\mathcal{T}_{\varepsilon,\lambda})$ is a metrizable topological module over 
the topological algebra $(L^{0}(\mathcal{F},\mathbb{K}),\mathcal{T}_{\varepsilon,\lambda})$, and it is also well known from \cite{FKV2009} that
$(E,\mathcal{T}_{c})$ is a topological module over the topological ring $(L^{0}(\mathcal{F},\mathbb{K}),\mathcal{T}_{c})$.

\begin{example}\label{ex.RN module}
	Let $L^{0}(\mathcal{F},\mathbb{K}^{d})$ be  the linear space of equivalence classes of random vectors from $(\Omega,\mathcal{F},P)$ to $\mathbb{K}^{d}$, then 
	$L^{0}(\mathcal{F},\mathbb{K}^{d})$ forms an $L^{0}$-module under the module multiplication
	$\cdot:L^{0}(\mathcal{F},\mathbb{K})\times L^{0}(\mathcal{F},\mathbb{K}^{d})$ defined by 
	$$\lambda \cdot x=(\lambda \xi_{1},\lambda \xi_{2},\cdots, \lambda \xi_{d}), \forall \lambda\in L^{0}(\mathcal{F},\mathbb{K}), \forall x=(\xi_{1},\xi_{2},\cdots, \xi_{d})\in L^{0}(\mathcal{F},\mathbb{K}^{d}).$$
	Define the $L^{0}$-inner product $\langle\cdot,\cdot\rangle:L^{0}(\mathcal{F},\mathbb{K}^{d})\times L^{0}(\mathcal{F},\mathbb{K}^{d})\rightarrow L^{0}(\mathcal{F},\mathbb{K})$ by 
	$$\langle x,y\rangle=\sum_{i=1}^{d}x_{i}\bar{y}_{i}, \forall x=(x_{1},x_{2},\cdots,x_{d}),y=(y_{1},y_{2},\cdots,y_{d})\in L^{0}(\mathcal{F},\mathbb{K}^{d}),$$
	where $\bar{y}_{i}$ denotes the complex conjugate of $y_{i}$ for each $i\in \{1,2,\cdots,n\}$. Then 
	$(L^{0}(\mathcal{F},\mathbb{K}^{d}),\langle\cdot,\cdot\rangle)$ is a $\mathcal{T}_{\varepsilon,\lambda}$-complete \textit{random inner product} module (briefly, a $\mathcal{T}_{\varepsilon,\lambda}$-complete $RIP$ module, see \cite{Guo1999} for details). 
	let $\|x\|=\sqrt{\langle x,x\rangle}$ for any $x\in L^{0}(\mathcal{F},\mathbb{K}^{d})$. Then $(L^{0}(\mathcal{F},\mathbb{K}^{d}),\|\cdot\|)$ is a $\mathcal{T}_{\varepsilon,\lambda}$-complete $RN$ module.
\end{example}

\par In fact, $L^{0}(\mathcal{F},\mathbb{K}^{d})$ is also $\mathcal{T}_{c}$-complete. Generally, according to Theorem 3.18 of \cite{Guo2010}, an 
$RN$ module is $\mathcal{T}_{\varepsilon,\lambda}$-complete iff it is  both $\mathcal{T}_{c}$-complete and $\sigma$-stable.

\par 
The randomized version of the Bolzano-Weierstrass theorem \cite{KS2001} states that for any a.s. bounded sequence $\{\xi^0_n, n\in \mathbb{N} \}$ of random vectors from 
$(\Omega, \mathcal{F}, P)$ to $\mathbb{R}^d$ (namely, $\sup \{|\xi^0_n(\omega)|: n\in \mathbb{N} \}< +\infty$ $a.s.$) there exists a sequence 
$\{n_{k}^{0}, k\in \mathbb{N}\}$ of positive integer-valued random variables on $(\Omega, \mathcal{F}, P)$ such that 
$\{n_{k}^{0}(\omega), k\in \mathbb{N} \}$ is strictly increasing for each $\omega\in \Omega$ and $\{\xi^0_{n_{k}^{0}}, k\in \mathbb{N} \}$ 
converges a.s., where $\xi^0_{n_{k}^{0}} (\omega)= \xi^0_{n_{k}^{0}(\omega) } (\omega)=\sum^{\infty} _{l=1} I_{(n_{k}^{0}=l)} (\omega) \xi^0_l(\omega)$ 
for any $k\in \mathbb{N}$ and any $\omega\in \Omega$ (as usual, $(n_{k}^{0}=l)$ denotes the set $\{\omega\in \Omega:n_{k}^{0}(\omega)=l\}$ 
for any $k$ and $l$ in $\mathbb{N}$). Motivated by this basic result, Guo et al. \cite{GWXYC2025} introduced the following notion of random 
sequential compactness in $RN$ modules.

\begin{definition}[\cite{GWXYC2025}]\label{defn.random sequentially compact}
	Let $(E, \|\cdot\|)$ be an $RN$ module over $\mathbb{K}$ with base $(\Omega, \mathcal{F}, P)$ and $G$ a 
	nonempty subset such that $G$ is contained in a $\sigma$-stable subset $H$ of $E$. Given a sequence $\{x_n, n\in \mathbb{N} \}$ 
	in $G$, a sequence $\{y_k, k\in \mathbb{N} \}$ in $H$ is called a random subsequence of $\{x_n, n\in \mathbb{N} \}$ 
	if there exists a sequence $\{n_k, k\in \mathbb{N} \}$ in $L^{0}(\mathcal{F},\mathbb{N})$ such that the following two conditions are satisfied:
	\begin{enumerate} 
		\item [(1)] $n_{k}< n_{k+1}$ on $\Omega$;
		\item [(2)] $y_k=x_{n_k}:= \sum^{\infty}_{l=1} I_{(n_k=l)} x_l$ for each $k\in \mathbb{N}$.
	\end{enumerate}
	Further, $G$ is said to be random (resp., relatively) sequentially compact if there exists a random 
	subsequence $\{y_k, k\in \mathbb{N}\}$ of $\{x_{n}, n\in \mathbb{N} \}$ for any sequence $\{x_{n}, n\in \mathbb{N} \}$ 
	in $G$ such that $\{y_k, k\in \mathbb{N}\}$ converges in $\mathcal{T}_{\varepsilon,\lambda}$ to some element in $G$ (resp., in $E$).
\end{definition}

\par Recall that a set $G$ of an $RN$ module with base $(\Omega,\mathcal{F},P)$ is said to be 
a.s. bounded if $\bigvee_{x\in G}\|x\|\in L_{+}^{0}(\mathcal{F})$.

\begin{remark}\label{rmk. random  sequentially compact}
	(1) In Definition \ref{defn.random sequentially compact}, we require $n_{k}$ to be an element in $L^{0}(\mathcal{F},\mathbb{N})$, 
	rather than a positive integer-valued random variable as in \cite[Definition 2.10]{GWXYC2025}. This choice is made for 
	simplifying written expressions, and it is easy to check that the two formulations are essentially equivalent. 
	(2) In the sense of Definition \ref{defn.random sequentially compact}, the randomized Bolzano-Weierstrass 
	theorem can be restated as: any a.s. bounded nonempty subset of $L^0(\mathcal{F}, \mathbb{R}^d)$ is random relatively sequentially compact. 
	Moreover, as pointed out in \cite{GWXYC2025}, if $G$ is $\sigma$-stable, then $G$ is random relatively sequentially compact iff 
	$\bar{G}_{\varepsilon,\lambda}$ (the closure of $G$ under $\mathcal{T}_{\varepsilon,\lambda}$) is random sequentially compact. Consequently, any a.s. bounded $\sigma$-stable nonempty 
	$\mathcal{T}_{\varepsilon,\lambda}$-closed subset of $L^0(\mathcal{F}, \mathbb{R}^d)$ (for example, an $L^{0}$-simplex) must be random  sequentially compact. 
	(3) For any $(n-1)$-dimensional  $L^{0}$-simplex $S$ of an $RN$ module, as pointed out by \cite[Lemma 5.5]{GWXYC2025}, there exists 
	an $L^{0}$-affine $\mathcal{T}_{\varepsilon,\lambda}$-isomorphism $T$ from $\Delta_{n-1}$ to $S$, where $\Delta_{n-1}:=Conv_{L^{0}}(\{e_{1},e_{2},\cdots,e_{n}\})$ 
	is the standard $(n-1)$-dimensional $L^{0}$-simplex of $L^0(\mathcal{F}, \mathbb{R}^n)$. Since $\Delta_{n-1}$ is random  sequentially compact as discussed above, and 
	$T$ is $\sigma$-stable and $\mathcal{T}_{\varepsilon,\lambda}$-continuous, then $S$ is random  sequentially compact.
\end{remark}

\begin{definition}[\cite{GWXYC2025}]\label{defn.continuous mapping}
	Let $(E_1, \|\cdot\|_1)$ and $(E_2, \|\cdot\|_2)$ be two $RN$ modules over the same scalar field $\mathbb{K}$ with 
	the same base $(\Omega, \mathcal{F}, P)$, and $G_1$ and $G_2$ be two nonempty subsets of $E_1$ and $E_2$, respectively. 
	A mapping $T:G_1\rightarrow G_2$ is said to be:
	\begin{enumerate}
		\item [(1)] $\mathcal{T}_{\varepsilon,\lambda}$-continuous if $T$ is a continuous mapping 
		from $(G_1, \mathcal{T}_{\varepsilon,\lambda})$ to $(G_2, \mathcal{T}_{\varepsilon,\lambda})$.
		\item [(2)] $\mathcal{T}_{c}$-continuous if $T$ is a continuous mapping 
		from $(G_1, \mathcal{T}_{c})$ to $(G_2, \mathcal{T}_{c})$.
		\item [(3)] $a.s.$ sequentially continuous at $x_0\in G_1$ if $\{\|T(x_n)-T(x_0)\|_2, n\in \mathbb{N} \}$ 
		converges $a.s.$ to $0$ for any sequence $\{x_n, n\in \mathbb{N}\}$ in $G_1$ such that $\{\|x_n-x_0\|, n\in \mathbb{N} \}$ 
		converges $a.s.$ to $0$. Further, $T$ is said to be $a.s.$ sequentially continuous if $T$ is $a.s.$ sequentially 
		continuous at any point in $G_1$.
		\item [(4)] random sequentially continuous at $x_0\in G_1$ if $G_1$ is $\sigma$-stable and for any sequence 
		$\{x_n, n\in \mathbb{N} \}$ in $G_1$ convergent in $\mathcal{T}_{\varepsilon,\lambda}$ to $x_0$ there exists 
		a random subsequence $\{x_{n_k}, k\in \mathbb{N}\}$ of $\{x_n, n\in \mathbb{N} \}$ such that $\{T(x_{n_k}), k\in \mathbb{N} \}$ 
		converges in $\mathcal{T}_{\varepsilon,\lambda}$ to $T(x_0)$. Further, $T$ is said to be random sequentially continuous if $T$ 
		is random sequentially continuous at any point in $G_1$.
	\end{enumerate}
\end{definition}

\begin{remark}\label{rmk.continuous}
	It is clear that a.s. sequential continuity implies $\mathcal{T}_{\varepsilon,\lambda}$-continuity. Moreover,  
	when $G_{1}$ is $\sigma$-stable, the $\mathcal{T}_{\varepsilon,\lambda}$-continuity of a mapping $T$ from $G_{1}$ to 
	$G_{2}$ necessarily implies random sequential continuity of $T$. Furthermore, \cite[Lemma 4.3]{GWXYC2025} shows that 
	a $\sigma$-stable mapping $T$ of $G_{1}$ to $G_{2}$ is random sequentially continuous iff it is $\mathcal{T}_{c}$-continuous.
\end{remark}

\par 
For any $(n-1)$-dimensional $L^{0}$-simplex  $S=Conv_{L^{0}}(\{x_{1},x_{2},\cdots,x_{n}\})$ in an $RN$ module $E$ and 
any  $L^{0}$-simplicial subdivision $\mathscr{S}=\{S_{i}:i\in I\}$ of $S$, the random diameter of $\mathscr{S}$ is defined by 
$$diam(\mathscr{S})=\bigvee_{i\in I}diam(S_{i}),$$
where $diam(S_{i}):=\bigvee\{\|x-y\|:x,y\in S_{i}\}$ is called the random diameter of $S_{i}$ for each $i\in I$.

\begin{lemma}\label{lemm.diam}
	Let $(E,\|\cdot\|)$ be an $RN$ module over $\mathbb{K}$ with base $(\Omega,\mathcal{F},P)$, 
	$S=Conv_{L^{0}}(\{x_{1},x_{2},\cdots,x_{n}\})$ an $(n-1)$-dimensional $L^{0}$-simplex in $E$
	and $\mathscr{S}$  an $L^{0}$-barycentric subdivision of $S$. 
	Then
	 $$diam(\mathscr{S})\leq \frac{n-1}{n}diam(S).$$
\end{lemma}
\begin{proof}
	First, we claim that  $diam(G)=\bigvee_{k,l=1}^{m}\|y_{i}-y_{j}\|$ for any $(n-1)$-dimensional $L^{0}$-simplex  
	$G=Conv_{L^{0}}(\{y_{1},y_{2},\cdots,y_{n}\})$ in $E$. In fact, it is clear that $\bigvee_{k,l=1}^{m}\|y_{k}-y_{l}\|\leq diam(G)$.
	On the other hand, for any $x=\sum_{i=1}^{m}\lambda_{i}y_{i}$ and $y=\sum_{i=1}^{m}\mu_{i}y_{i}$ in $G$, where 
	$\lambda_{i},\mu_{i}\in L_{+}^{0}(\mathcal{F}) $ for any $i\in \{1,2,\cdots,m\}$ and $ \sum_{i=1}^{m}\lambda_{i}=\sum_{i=1}^{m}\mu_{i}=1$, we have 
	\begin{align}
		\|x-y\|&=\|\sum_{i=1}^{m}\lambda_{i}y_{i}-\sum_{i=1}^{m}\mu_{i}y_{i}\|\nonumber\\
		&=\|\sum_{i,j=1}^{m}\lambda_{i}\mu_{j}(y_{i}-y_{j})\|\nonumber\\
		&\leq\sum_{i,j=1}^{m}\lambda_{i}\mu_{j}(\bigvee_{k,l=1}^{m}\|y_{k}-y_{l}\|)\nonumber\\
		&=\bigvee_{k,l=1}^{m}\|y_{k}-y_{l}\|.\nonumber
	\end{align}
	
	\par Second, for any $\pi\in G(n)$ and any $k,l\in \{1,2,\cdots,n\}$, without loss of generality,  we can assume that 
	$k>l$. Then
	\begin{align}
		\|y_{k}^{\pi}-y_{l}^{\pi}\|&=\|\frac{1}{k}\sum_{i=1}^{k}x_{\pi(i)}-\frac{1}{l}\sum_{j=1}^{l}x_{\pi(j)}\|\nonumber\\
		&=\frac{k-l}{k}\|\frac{1}{l}\sum_{i=1}^{l}x_{\pi(i)}-\frac{1}{k-l}\sum_{j=l+1}^{k}x_{\pi(j)}\|\nonumber\\
		&\leq \frac{n-1}{n}\|\frac{1}{l}\sum_{i=1}^{l}x_{\pi(i)}-\frac{1}{k-l}\sum_{j=l+1}^{k}x_{\pi(j)}\|\nonumber\\
		&\leq \frac{n-1}{n}diam(S),\nonumber
	\end{align} 
	implying
	$$diam(S_{\pi})=\bigvee_{k,l=1}^{n}\|y_{k}^{\pi}-y_{l}^{\pi}\|\leq \frac{n-1}{n}diam(S).$$
	It follows that 
	$$diam(\mathscr{S})=\bigvee_{\pi\in G(n)}diam(S_{\pi})\leq \frac{n-1}{n}diam(S).$$
\end{proof}

\par 
The key idea of the proof for the classical Brouwer fixed point theorem for a continuous self-mapping on a simplex $S'$ lies in constructing  a proper labeling function $\psi$ on a given simplicial subdivision $\mathscr{S}'$ of $S'$  that is compatible with the continuous self-mapping on $S'$, see \cite[Theorem 6.1]{Bor1985} for details. Lemma \ref{lemm.label} below provides a method for constructing such a proper $L^{0}$-labeling function in connection with an $(n-1)$-dimensional $L^{0}$-simplex $S$ in an $RN$ module that improves the formulation of \cite[Lemma 2.2]{DKKS2013} in that the single proper $L^{0}$-labeling function provided here can be applied to the $L^{0}$-simplicial subdivision of every $(n-1)$-dimensional $L^{0}$-simplex contained in $S$.

\begin{lemma}\label{lemm.label}
	Let $(E,\|\cdot\|)$ be an $RN$ module with base $(\Omega,\mathcal{F},P)$, 
	$S=Conv_{L^{0}}(\{x_{1},x_{2},\cdots,x_{n}\})$ an $(n-1)$-dimensional $L^{0}$-simplex in $E$
	and $f:S\rightarrow S$ be a $\sigma$-stable mapping. Then there exists a mapping 
	$\phi:S\rightarrow L^{0}(\mathcal{F},\{1,2,\cdots,n\})$ such that the following conditions are satisfied:
	\begin{enumerate}
		\item [(1)] For any $x=\sum_{i=1}^{n}\lambda_{i}^{x}x_{i}\in S$, further let $f(x)=\sum_{i=1}^{n}\mu_{i}^{x}x_{i}$,  where 
		$\lambda_{i}^{x},\mu_{i}^{x}\in L_{+}^{0}(\mathcal{F})$ for any $i\in \{1,2,\cdots,n\}$ and $\sum_{i=1}^{n}\lambda_{i}^{x}=\sum_{i=1}^{n}\mu_{i}^{x}=1$. 
		Then we have
		$$(\phi(x)=i)\wedge [(\lambda_{i}^{x}=0)\vee (\lambda_{i}^{x}<\mu_{i}^{x})]=0, \forall i\in \{1,2,\cdots,n\}.$$ 
		\item [(2)] For any  $(n-1)$-dimensional $L^{0}$-simplex $S_{1}$ contained in $S$ and any $L^{0}$-simplicial subdivision $\mathscr{S}$ of $S_{1}$, $\phi|_{ext_{L^{0}}(\mathscr{S})}$ is a proper $L^{0}$-labeling function of $\mathscr{S}$.
	\end{enumerate}
\end{lemma}
\begin{proof}
	Since for any $x\in S$ there exists a unique $\{\lambda_{i}^{x},i=1\sim n\}\subseteq L^{0}(\mathcal{F},[0,1])$ such that $x=\sum_{i=1}^{n}\lambda_{i}^{x}x_{i}$
	and $\sum_{i=1}^{n}\lambda_{i}^{x}=1$, where $(\lambda_{1}^{x},\lambda_{2}^{x},\cdots,\lambda_{n}^{x})$ is called the $L^{0}$-barycentric coordinate of $x$. For brevity and clarity, in the proof, 
	we always use $\lambda_{i}^{x}$ for the $i$-th $L^{0}$-barycentric coordinate of $x$ and $\mu_{i}^{x}$ for the $i$-th $L^{0}$-barycentric coordinate of $f(x)$.
	Then, for any $x\in S$,  $x=\sum_{i=1}^{n}\lambda_{i}^{x}x_{i}$ and $f(x)=\sum_{i=1}^{n}\mu_{i}^{x}x_{i}$,
	we claim  that
	$$\bigvee_{i=1}^{n}[(\lambda_{i}^{x}>0)\wedge (\lambda_{i}^{x}\geq \mu_{i}^{x})]=1.$$
	Otherwise, 
	$$\bigwedge_{i=1}^{n}[(\lambda_{i}^{x}=0)\vee (\lambda_{i}^{x}< \mu_{i}^{x})]>0 .$$
	Since $\bigvee_{i=1}^{n}(\lambda_{i}^{x}>0)=1$, there exists $j\in \{1,2,\cdots,n\}$ such that 
	$$(\lambda_{j}^{x}>0)\wedge \{\bigwedge_{i=1}^{n}[(\lambda_{i}^{x}=0)\vee (\lambda_{i}^{x}< \mu_{i}^{x})]\}>0,$$
	namely,
	$$a:=[(\lambda_{j}^{x}>0)\wedge (\lambda_{j}^{x}<\mu_{j}^{x})]\wedge\{\bigwedge_{i\in \{1,2,\cdots,n\}\backslash\{j\}}[(\lambda_{i}^{x}=0)\vee (\lambda_{i}^{x}< \mu_{i}^{x})]\} >0.$$
	It follows that $I_{a}\lambda_{j}^{x}<I_{a}\mu_{j}^{x}$ and 
	$$I_{a}\lambda_{i}^{x}\leq I_{a}\mu_{i}^{x}, \forall i\in \{1,2,\cdots,n\}\backslash\{j\},$$
	contradicting the facts that $\sum_{i=1}^{n}\lambda_{i}^{x}=\sum_{i=1}^{n}\mu_{i}^{x}=1$. 
	
	\par For any $x\in S$, let 
	$$b_{i}^{x}=(\lambda_{i}^{x}>0)\wedge (\lambda_{i}^{x}\geq \mu_{i}^{x}), \forall i\in \{1,2,\cdots,n\},$$ 
	and further let 
	$a_{1}^{x}=b_{1}^{x}$,
	$$a_{i}^{x}=b_{i}^{x}\wedge(\bigvee_{j=1}^{i-1}b_{j}^{x})^{c}, \forall i\in \{2,3,\cdots,n\}$$ 
	and $\phi(x)=\sum_{i=1}^{n}I_{a_{i}^{x}}i$. Since each $a_{i}^{x}$ is uniquely determined by $x$, $f$ and $S$, the mapping  $\phi$ defined in this 
	manner is well-defined. Next, we will show that $\phi$ is the desired mapping.
	
	\par (1) For any $x\in S$, it is clear that $\phi(x)\in L^{0}(\mathcal{F},\{1,2,\cdots,n\})$  and  
	$$(\phi(x)=i)\wedge [(\lambda_{i}^{x}=0)\vee (\lambda_{i}^{x}<\mu_{i}^{x})]=a_{i}^{x}\wedge [(\lambda_{i}^{x}=0)\vee (\lambda_{i}^{x}<\mu_{i}^{x})]=0$$
	for any $i\in \{1,2,\cdots,n\}$.
	
	\par (2) For any $\{c_{k},k\in\mathbb{N}\}\in p(1)$ and any sequence $\{y_{k},k\in \mathbb{N}\}$ in $ext_{L^{0}}(\mathscr{S})$, since 
	$y_{k}=\sum_{i=1}^{n}\lambda_{i}^{y_{k}}x_{i}$ and $f(y_{k})=\sum_{i=1}^{n}\mu_{i}^{y_{k}}x_{i}$ for each $ k\in \mathbb{N}$, we have 
	$$y:=\sum_{k=1}^{\infty}I_{c_{k}}y_{k}=\sum_{k=1}^{\infty}I_{c_{k}}(\sum_{i=1}^{n}\lambda_{i}^{y_{k}}x_{i})=\sum_{i=1}^{n}(\sum_{k=1}^{\infty}I_{c_{k}}\lambda_{i}^{y_{k}})x_{i}$$
	and 
	$$f(y):=\sum_{k=1}^{\infty}I_{c_{k}}f(y_{k})=\sum_{k=1}^{\infty}I_{c_{k}}(\sum_{i=1}^{n}\mu_{i}^{y_{k}}x_{i})=\sum_{i=1}^{n}(\sum_{k=1}^{\infty}I_{c_{k}}\mu_{i}^{y_{k}})x_{i},$$
	then $\lambda_{i}^{y}=\sum_{k=1}^{\infty}I_{c_{k}}\lambda_{i}^{y_{k}}$, $\mu_{i}^{y}=\sum_{k=1}^{\infty}I_{c_{k}}\mu_{i}^{y_{k}}$ and 
	\begin{align}
		b_{i}^{y}&=(\lambda_{i}^{y}>0)\wedge (\lambda_{i}^{y}\geq \mu_{i}^{y})\nonumber\\
		&=[\bigvee_{k=1}^{\infty} (c_{k}\wedge (\lambda_{i}^{y_{k}}>0))]\wedge [\bigvee_{k=1}^{\infty} (c_{k}\wedge (\lambda_{i}^{y_{k}}\geq \mu_{i}^{y_{k}}))] \nonumber\\
		&=\bigvee_{k=1}^{\infty} [c_{k}\wedge (\lambda_{i}^{y_{k}}>0)\wedge(\lambda_{i}^{y_{k}}\geq \mu_{i}^{y_{k}})]\nonumber\\
		&=\bigvee_{k=1}^{\infty} (c_{k}\wedge b_{i}^{y_{k}})\nonumber
	\end{align}
	for any $i\in \{1,2,\cdots,n\}$. It follows that 
	\begin{align}
		a_{i}^{y}&=b_{i}^{y}\wedge(\bigvee_{j=1}^{i-1}b_{j}^{y})^{c}\nonumber\\
		&=[\bigvee_{k=1}^{\infty} (c_{k}\wedge b_{i}^{y_{k}})]\wedge \{\bigvee_{j=1}^{i-1}[\bigvee_{l=1}^{\infty} (c_{l}\wedge b_{j}^{y_{l}})]\}^{c}\nonumber\\
		&=\bigvee_{k=1}^{\infty}\{\bigwedge_{j=1}^{i-1}[\bigwedge_{l=1}^{\infty} (c_{k}\wedge b_{i}^{y_{k}})\wedge(c_{l}\wedge b_{j}^{y_{l}})^{c}]\}\nonumber\\
		&=\bigvee_{k=1}^{\infty}\{\bigwedge_{j=1}^{i-1}[(c_{k}\wedge b_{i}^{y_{k}})\wedge(c_{k}\wedge b_{j}^{y_{k}})^{c} \wedge (\bigwedge_{l\in \mathbb{N},l\neq k} (c_{k}\wedge b_{i}^{y_{k}})\wedge(c_{l}\wedge b_{j}^{y_{l}})^{c})]\}\nonumber\\
		&=\bigvee_{k=1}^{\infty}\{\bigwedge_{j=1}^{i-1}[(c_{k}\wedge b_{i}^{y_{k}})\wedge(b_{j}^{y_{k}})^{c} \wedge (c_{k}\wedge b_{i}^{y_{k}})]\}\nonumber\\
		&=\bigvee_{k=1}^{\infty} [c_{k}\wedge b_{i}^{y_{k}}\wedge (\bigvee_{j=1}^{i-1}b_{j}^{y_{k}})^{c}]\nonumber\\
		&=\bigvee_{k=1}^{\infty}(c_{k}\wedge a_{i}^{y_{k}})\nonumber
	\end{align}
	for any $i\in \{1,2,\cdots,n\}$. Thus, we have 
	\begin{align}
		\phi(\sum_{k=1}^{\infty}I_{c_{k}}y_{k})&=\phi(y)\nonumber\\
		&=\sum_{i=1}^{n}I_{a_{i}^{y}}i\nonumber\\
		&=\sum_{i=1}^{n}I_{\bigvee_{k=1}^{\infty}(c_{k}\wedge a_{i}^{y_{k}})}i\nonumber\\
		&=\sum_{i=1}^{n}\sum_{k=1}^{\infty}I_{c_{k}}I_{a_{i}^{y_{k}}}i\nonumber\\
		&=\sum_{k=1}^{\infty}I_{c_{k}}(\sum_{i=1}^{n}I_{a_{i}^{y_{k}}}i)\nonumber\\
		&=\sum_{k=1}^{\infty}I_{c_{k}} \phi (y_{k}),\nonumber
	\end{align}
	which implies that $\phi$ is $\sigma$-stable, namely, $\phi$ is an $L^{0}$-labeling mapping of $\mathscr{S}$.
	
	\par Moreover, for any $x\in ext_{L^{0}}(\mathscr{S})$, we have 
	\begin{align}
		(\phi(x)=i)\wedge (\lambda_{i}^{x}=0)&=a_{i}^{x}\wedge (\lambda_{i}^{x}=0)\nonumber\\
		&\leq b_{i}^{x}\wedge (\lambda_{i}^{x}=0)\nonumber\\
		&=(\lambda_{i}^{x}>0)\wedge (\lambda_{i}^{x}\geq \mu_{i}^{x})\wedge (\lambda_{i}^{x}=0)\nonumber\\
		&=0\nonumber
	\end{align}
	for any $i\in \{1,2,\cdots,n\}$. Hence $\phi$ is proper.
\end{proof}

\par Proposition \ref{prop.converge a.s.3} below will be used in the proof of Lemma \ref{lemm.Brouwer}.

	\begin{proposition}[{\cite[Proposition 2.14]{TMGYS2025}}]\label{prop.converge a.s.3}
		Let $(E_{1},\|\cdot\|)$ and $(E_{2},\|\cdot\|)$ be two $RN$ modules over $\mathbb{K}$ with base $(\Omega,\mathcal{F},P)$, $G_{1}\subset E_{1}$ and $G_{2}\subset E_{2}$ two $\sigma$-stable subsets, and $f: G_{1}\rightarrow G_{2}$ a $\sigma$-stable random sequentially continuous mapping. For a sequence $\{(x_{1}^{m}, x_{2}^{m},\cdots, x_{l}^{m}), m\in \mathbb{N} \}$ in $G_{1}^{l}$, where $l$ is a fixed positive integer and $G^{l}_{1}$ is the $l$--th Cartesian power set of $G_{1}$, if there exists a random subsequence $\{(x_{1}^{M_{n}^{(0)}}, x_{2}^{M_{n}^{(0)}}, \cdots, x_{l}^{M_{n}^{(0)}}), n\in \mathbb{N}\}$ of which such that $\{x_{i}^{M_{n}^{(0)}}, n\in \mathbb{N}\}$ converges in $\mathcal{T}_{\varepsilon,\lambda}$ to some $y_{i}\in G_{1}$ for each $i\in \{1,2,\cdots, l\}$, then there exists a random subsequence $\{(x_{1}^{M_{n}}, x_{2}^{M_{n}}, \cdots, x_{l}^{M_{n}}),n\in \mathbb{N}\}$ of $\{(x_{1}^{M_{n}^{(0)}}, x_{2}^{M_{n}^{(0)}}, \cdots, x_{l}^{M_{n}^{(0)}}), n\in \mathbb{N}\}$ such that $\{x_{i}^{M_{n}}, n\in \mathbb{N}\}$ converges in $\mathcal{T}_{\varepsilon,\lambda}$ to $y_i$ and $\{f(x_{i}^{M_{n}}),n\in \mathbb{N}\}$ converges in $\mathcal{T}_{\varepsilon,\lambda}$ to $f(y_{i})$ for each $i\in \{1,2,\cdots, l\}$.
\end{proposition}

\par Lemma \ref{lemm.Brouwer} below  is the $L^{0}$-simplex version of the random Brouwer fixed point theorem in $RN$ modules.

\begin{lemma}\label{lemm.Brouwer}
	Let $S=Conv_{L^{0}}(\{x_{1},x_{2},\cdots,x_{n}\})$ be an $(n-1)$-dimensional $L^{0}$-simplex in an $RN$ module $E$ and 
	$f:S\rightarrow S$ be a $\sigma$-stable and random sequentially continuous mapping. Then $f$ has a fixed point.
\end{lemma}
\begin{proof}
	Let $\phi$ be the mapping obtained as in Lemma \ref{lemm.label}.  
	
	\par Let $\mathscr{S}$ be an $L^{0}$-barycentric subdivision of $S$. For $\phi$ and $\mathscr{S}$, by Theorem \ref{thm.Sperner} there exists an $(n-1)$-dimensional $L^{0}$-simplex 
	$S^{1}:=Conv_{L^{0}}(\{y_{1}^{1},y_{2}^{1},\cdots,y_{n}^{1}\})$,  which is a finite concatenation of some finite family of $L^{0}$-simplexes in $\mathscr{S}$,  namely, there exist a finite family $\{S_{k}^{1},k=1\sim M\}$ 
	of $(n-1)$-dimensional $L^{0}$-simplexes in $\mathscr{S}$ and $\{a_{k}^{1},k=1\sim M\}\in p(1)$ such that 
	$$S^{1}=\sum_{k=1}^{M}I_{a_{k}^{1}}S_{k}^{1}$$ 
	is  completely labeled by $\phi$ and satisfies that
	$$\phi(y_{j}^{1})=j~\text{for~any}~j\in \{1,2,\cdots,n\}.$$
	It is easy to see that 
	\begin{align}
		diam(S^{1})&\leq \bigvee_{k=1}^{M}diam(S_{k}^{1})\nonumber\\
		&\leq diam (\mathscr{S})\nonumber\\
		&\leq \frac{n-1}{n}diam(S).\nonumber
	\end{align}
	
	\par Let $\mathscr{S}_{1}$ be an $L^{0}$-barycentric subdivision of $S^{1}$. Now, similarly, for $\phi$ and $\mathscr{S}_{1}$, by Theorem \ref{thm.Sperner} there exists an $(n-1)$-dimensional $L^{0}$-simplex 
	$S^{2}:=Conv_{L^{0}}(\{y_{1}^{2},y_{2}^{2},\cdots,y_{n}^{2}\})$,  which is a finite concatenation of some finite family of $L^{0}$-simplexes in $\mathscr{S}_{1}$, such that 
	$S^{2}$ is  completely labeled by $\phi$
	 and satisfies that
	$$\phi(y_{j}^{2})=j~\text{for~any}~j\in \{1,2,\cdots,n\},$$
	and 
	$$diam(S^{2})\leq  \frac{n-1}{n}diam(S^{1})\leq (\frac{n-1}{n})^{2}diam(S).$$

    \par 
	By induction, for each $m\in \mathbb{N}$, there exists 
	an $(n-1)$-dimensional $L^{0}$-simplex $S^{m}:=Conv_{L^{0}}(\{y_{1}^{m},y_{2}^{m},\cdots,y_{n}^{m}\})$, which is a finite concatenation of some finite family of $L^{0}$-simplexes in $\mathscr{S}_{m-1}$, such that 
	$S^{m}$ is  completely labeled by $\phi$  and satisfies that
	$$\phi(y_{j}^{m})=j~\text{for~any}~j\in \{1,2,\cdots,n\},$$
	and  
	$$diam(S^{m})\leq  (\frac{n-1}{n})^{m}diam(S),$$
	where $\mathscr{S}_{m}$ is the  $L^{0}$-barycentric subdivision of $S^{m}$ for each $m\in \mathbb{N}$.

    \par 
	Then $\{(y_{1}^{m},y_{2}^{m}\cdots,y_{n}^{m}),m\in \mathbb{N}\}$ is a sequence in $S^{(n)}$, where $S^{(n)}$ is the 
	$n$-th Cartesian power set of $S$. 
	
	\par For $\{y_{1}^{m},m\in \mathbb{N}\}$, since $S$ is random sequentially compact by Remark \ref{rmk. random  sequentially compact}, there exists a random subsequence $\{y_{1}^{m_{k}'}, k\in \mathbb{N}\}$ converging in $\mathcal{T}_{\varepsilon,\lambda}$ to some $y$ in $S$. Since $\{diam(S^{m}),m\in \mathbb{N}\}$  converges $a.s.$ to $0$, by \cite[Lemma 2.11]{TMGYS2025} $\{d^{m_{k}'}, k\in \mathbb{N}\}$ converges $a.s.$ (of course, also in $\mathcal{T}_{\varepsilon,\lambda}$) to $0$, where 
	$$d^{m_{k}'}:=\sum_{l=1}^{\infty}I_{( m_{k}' =l)}diam(S^{l}), \forall k\in \mathbb{N}.$$
	For any $i\in \{2,3,\cdots,n\}$, we have 
	\begin{align}
		\|y_{i}^{m_{k}'}-y\|&\leq \|y_{i}^{m_{k}'}-y_{1}^{m_{k}'}\|+\|y_{1}^{m_{k}'}-y\|\nonumber\\
		&=\sum_{l=1}^{\infty}I_{( m_{k}' =l)}\|y_{i}^{l}-y_{1}^{l}\|+\|y_{1}^{m_{k}'}-y\|\nonumber\\
		&\leq \sum_{l=1}^{\infty}I_{( m_{k}' =l)}diam(S^{l})+\|y_{1}^{m_{k}'}-y\|\nonumber\\
		&= d^{m_{k}'}+\|y_{1}^{m_{k}'}-y\|,\nonumber 
	\end{align}
	implying that  $\{y_{i}^{m_{k}'}, k\in \mathbb{N}\}$ converges in $\mathcal{T}_{\varepsilon,\lambda}$ to $y$. 
	To sum up,  $\{(y_{1}^{m_{k}'},y_{2}^{m_{k}'},\cdots,y_{n}^{m_{k}'})$, $k\in \mathbb{N}\}$ is a random subsequence of 
	$\{(y_{1}^{m},y_{2}^{m},\cdots,y_{n}^{m}),m\in \mathbb{N}\}$ such that $\{y_{i}^{m_{k}'}, k\in \mathbb{N}\}$ converges in 
	$\mathcal{T}_{\varepsilon,\lambda}$ to $y$ for any $i\in \{1,2,\cdots,n\}$. Since $f$ is $\sigma$-stable and random sequentially continuous,
	by Proposition \ref{prop.converge a.s.3} there exists a random subsequence $\{(y_{1}^{m_{k}},y_{2}^{m_{k}},\cdots,y_{n}^{m_{k}}),k\in \mathbb{N}\}$ 
	of $\{(y_{1}^{m_{k}'},y_{2}^{m_{k}'},\cdots,y_{n}^{m_{k}'}),k\in \mathbb{N}\}$ such that $\{y_{i}^{m_{k}},k\in \mathbb{N}\}$ converges in $\mathcal{T}_{\varepsilon,\lambda}$ 
	to $y$ and $\{f(y_{i}^{m_{k}}),k\in \mathbb{N}\}$ converges in $\mathcal{T}_{\varepsilon,\lambda}$ to $f(y)$ for each $i\in \{1,2,\cdots,n\}$.
	
	\par Let $y=\sum_{i=1}^{n}\lambda_{i}x_{i}$ and $f(y)=\sum_{i=1}^{n}\mu_{i}x_{i}$, where  
	$\lambda_{i},\mu_{i}\in L_{+}^{0}(\mathcal{F})$ for any $i\in \{1,2,\cdots,n\}$ are such that $\sum_{i=1}^{n}\lambda_{i}=\sum_{i=1}^{n}\mu_{i}=1$.
	Next, we will show that $f(y)=y$, namely, $y$ is a fixed point of $f$.
	
	\par For any fixed $j\in \{1,2,\cdots,n\}$ and $k\in \mathbb{N}$, let $y_{j}^{m_{k}}=\sum_{i=1}^{n}\lambda_{i}^{m_{k},j}x_{i}$ and 
	$f(y_{j}^{m_{k}})=\sum_{i=1}^{n}\mu_{i}^{m_{k},j}x_{i}$, where $\lambda_{i}^{m_{k},j},\mu_{i}^{m_{k},j}\in L_{+}^{0}(\mathcal{F})$ 
	for any $i\in \{1,2,\cdots,n\}$ and $\sum_{i=1}^{n}\lambda_{i}^{m_{k},j}=\sum_{i=1}^{n}\mu_{i}^{m_{k},j}=1$.
	Since $\phi$ is $\sigma$-stable, we have 
	\begin{align}
		\phi(y_{j}^{m_{k}})&=\phi(\sum_{l=1}^{\infty}I_{(m_{k}=l)}y_{j}^{m'_{l}})\nonumber\\
		&=\sum_{l=1}^{\infty}I_{(m_{k}=l)}\phi(y_{j}^{m'_{l}})\nonumber\\
		&=\sum_{l=1}^{\infty}I_{(m_{k}=l)}\phi(\sum_{r=1}^{\infty}I_{(m'_{l}=r)}y_{j}^{r}))\nonumber\\
		&=\sum_{l=1}^{\infty}I_{(m_{k}=l)}\sum_{r=1}^{\infty}I_{(m'_{l}=r)}\phi(y_{j}^{r})\nonumber\\
		&=\sum_{l=1}^{\infty}I_{(m_{k}=l)}\sum_{r=1}^{\infty}I_{(m'_{l}=r)}j\nonumber\\
		&=j.\nonumber
	\end{align}
	So $\lambda_{j}^{m_{k},j}\geq \mu_{j}^{m_{k},j}$ by Lemma \ref{lemm.label}. 
	
	\par Furthermore,  by \cite[Lemma 5.5]{GWXYC2025}, 
	$\{\lambda_{j}^{m_{k},j},k\in \mathbb{N}\}$ converges in $\mathcal{T}_{\varepsilon,\lambda}$ to $\lambda_{j}$ and 
	$\{\mu_{j}^{m_{k},j},k\in \mathbb{N}\}$ converges in $\mathcal{T}_{\varepsilon,\lambda}$ to $\mu_{j}$ for each $j\in \{1,2,\cdots,n\}$, 
	then  
	$$\lambda_{j}\geq \mu_{j},~ \forall j\in \{1,2,\cdots,n\}.$$
	However, $\sum_{j=1}^{n}\lambda_{j}=\sum_{j=1}^{n}\mu_{j}=1$, then 
	$$\lambda_{j}= \mu_{j},~ \forall j\in \{1,2,\cdots,n\}.$$
	Hence $f(y)=y$, and $y$ is a fixed point.
\end{proof}

\begin{remark}\label{rmk.Brouwer1}
  Lemma \ref{lemm.Brouwer} is a generalization of \cite[Theorem 2.3]{DKKS2013}, since $L^{0}(\mathcal{F},\mathbb{R}^{d})$ is an $RN$ module, and a $\sigma$-stable a.s. sequentially continuous mapping is, obviously, random sequentially continuous, as noted in Remark \ref{rmk.continuous}. Since random sequential continuity is much more complicated than a.s. sequential continuity,  Proposition \ref{prop.converge a.s.3} has played a key role in the treatment of random sequential continuity in the proof of Lemma \ref{lemm.Brouwer}. Besides, we would also like to emphasize another difference between the proofs of Lemma \ref{lemm.Brouwer} and \cite[Theorem 2.3]{DKKS2013}: the sequence $\{S^{m},m\in \mathbb{N}\}$ of $(n-1)$-dimensional $L^{0}$-simplexes, which is constructed in Lemma \ref{lemm.Brouwer}, has the nice property that $S^{m+1}\subseteq S^{m}$ for each $m\in \mathbb{N}$. Contrarily, the sequence $\{S^{m},m\in \mathbb{N}\}$ constructed in \cite[Theorem 2.3]{DKKS2013} does not have this property. The nice property of our sequence will be used to study other problems in the future. Besides, the proof of Lemma \ref{lemm.Brouwer} is simpler and clearer than that of \cite[Theorem 2.3]{DKKS2013} since we only employ a single proper $L^{0}$-labeling function throughout the proof of Lemma \ref{lemm.Brouwer}.  
\end{remark}

\par Let $H$ be a Hilbert space and $G$ be a nonempty closed convex subset of $H$. Then it is well known that the classical projection $P_{G}$ of $H$ onto $G$
is nonexpensive. By a completely similar method, one can prove the following lemma.

\begin{lemma}\label{lemm.projection}
	Let $(H,\langle\cdot,\cdot\rangle)$ be a $\mathcal{T}_{\varepsilon,\lambda}$-complete random inner product module over $\mathbb{K}$ 
	with base $(\Omega,\mathcal{F},P)$ and $G$ be a nonempty $\mathcal{T}_{\varepsilon,\lambda}$-closed $L^{0}$-convex subset of $H$. 
	Then for any $x\in H$, there exists an unique element $P_{G}(x)$ in $G$ such that 
	$$\|x-P_{G}(x)\|=\bigwedge \{\|x-g\|:g\in G\},$$ 
	where $\|\cdot\|$ is the $L^{0}$-norm induced by the $L^{0}$-inner product $\langle\cdot,\cdot\rangle$. Furthermore, this defines 
	a mapping $P_{G}:H\rightarrow G$, called the random projection onto $G$ that is nonexpensive, namely, 
	$$\|P_{G}(x)-P_{G}(y)\|\leq \|x-y\|, \forall x,y\in H.$$
\end{lemma}

\par In Lemma \ref{lemm.projection}, since $G$ is $\sigma$-stable and $P_{G}$ is $L^{0}$-Lipschitzian, it follows that $P_{G}$ is $\sigma$-stable by \cite[Lemma 2.11]{GZWG2020}.

\par We can now prove Theorem \ref{thm.Brouwer}.

\begin{proof}[Proof of Theorem \ref{thm.Brouwer}]
	Since $G$ is $a.s.$ bounded, there exists an $L^{0}$-simplex $S$ in $L^0(\mathcal{F},\mathbb{R}^d)$ such that $G\subseteq S$. 
	Since $L^0(\mathcal{F},\mathbb{R}^d)$ is a $\mathcal{T}_{\varepsilon,\lambda}$-complete random inner product module, let $h:S\rightarrow G$ be the restriction of 
	$P_{G}$ to $S$, where $P_{G}$ is the random projection of $L^0(\mathcal{F},\mathbb{R}^d)$ onto $G$. It is clear that $h$ is $\mathcal{T}_{c}$-continuous.

	\par 
	Consider the mapping $\hat{f}: S \rightarrow G\subseteq S$ defined by 
	$$\hat{f}(x)= (f \circ h)(x), \forall x\in S.$$ 
	Since $f$ is $\sigma$-stable and $\mathcal{T}_{c}$-continuous, it follows that $\hat{f}$ is also $\sigma$-stable and $\mathcal{T}_{c}$-continuous.  By Remark \ref{rmk.continuous},
	$\hat{f}$ is $\sigma$-stable and random sequentially continuous, then, by Lemma \ref{lemm.Brouwer} there exists $z\in S$ such that $\hat{f}(z)=z$, 
	which implies that $z\in G$. By the definition of $h$,  we have $h(z)=z$, and hence $f(z)=z$.
\end{proof}

\par Proposition \ref{prop.Brouwer} below was first established as \cite[Lemma 4.7]{GWXYC2025}, which is merely a reformulation of 
Theorem \ref{thm.Brouwer} in the case of a finitely generated $RN$ module. For the proof of Proposition \ref{prop.Brouwer}, we refer the reader to \cite{GWXYC2025}.

\begin{proposition}\label{prop.Brouwer}
	Let $G$ be an $L^0$-convex, $\mathcal{T}_{\varepsilon,\lambda}$-closed and $a.s.$ bounded subset of a finitely 
	generated $RN$ module $(E,\|\cdot\| ) $ over $\mathbb{K}$ with base $(\Omega, \mathcal{F}, P)$ and $f: G \rightarrow G$ 
	a $\sigma$-stable and random sequentially continuous mapping, where we recall that $E$ is finitely generated 
	if there exists a finite subset of $E$, denoted by $\{x_1,x_2, \cdots, x_d\}$ for some $d\in \mathbb{N}$, 
	such that $E= \{\sum^d_{i=1} \xi_i \cdot x_i: \xi_i\in L^0(\mathcal{F}, \mathbb{K}) $ for each $i\in \{1,2, \cdots, d\}\}$. 
	Then $f$ has a fixed point.
\end{proposition}

\section{The equivalence between Proposition \ref{prop.2.3} and Proposition \ref{prop.2.4}}\label{sec.5}

The main purpose of this section is to establish the equivalence between Proposition \ref{prop.2.3} and Proposition \ref{prop.2.4}. In fact, that Proposition \ref{prop.2.3} implies  Proposition \ref{prop.2.4} is obvious, but the reverse implication is another matter. Owing to the profound result of Ponosov \cite{Ponosov1987}, which is the solution to a Nemytskij conjecture (see also Lemma \ref{lemm.5.2} below), we can prove the reverse implication. It is in proving the reverse implication that we also obtain the following two somewhat surprising by-products --- Theorem \ref{thm.equivalent2} and Proposition \ref{prop.5.8}. Theorem \ref{thm.equivalent2} shows that both Proposition \ref{prop.2.3} and Proposition \ref{prop.2.4} are equivalent to the random version of the original Brouwer fixed point theorem (namely, part (2) of Theorem \ref{thm.equivalent2}, whose proof only needs the use of the measurable selection theorem --- \cite[Theorem 4.2]{Wagner1977}). Proposition \ref{prop.5.8} shows that the $\mathcal{T}_{\varepsilon,\lambda}$-continuity and a.s. sequential continuity coincide for a $\sigma$-stable mapping from a $\mathcal{T}_{\varepsilon,\lambda}$-closed $L^{0}$-convex subset of $L^{0}(\mathcal{F},\mathbb{R}^{d})$ to $L^{0}(\mathcal{F},\mathbb{R}^{d})$.
 

\begin{lemma}\label{lemm.5.1}
	Let $G$ be a $\sigma$-stable subset of $L^{0}(\mathcal{F},\mathbb{R}^{d})$. Then a mapping $h:G\rightarrow G$ is local iff $h$ is $\sigma$-stable.
\end{lemma}
\begin{proof}
	\textbf{Necessity}. For any sequence $\{x_{n},n\in \mathbb{N}\}$ in $G$ and any countable partition $\{a_{n},n\in \mathbb{N}\}$ of unity in $B_{\mathcal{F}}$, then, according to 
	the local property of $h$, $I_{a_{n}}h(\sum_{n=1}^{\infty}I_{a_{n}}x_{n})=I_{a_{n}}h(x_{n})$, so that 
	$$h(\sum_{n=1}^{\infty}I_{a_{n}}x_{n})=(\sum_{n=1}^{\infty}I_{a_{n}}) h(\sum_{n=1}^{\infty}I_{a_{n}}x_{n})=\sum_{n=1}^{\infty}I_{a_{n}}h(x_{n}).$$
	
	\par \textbf{Sufficiency}. For any $a\in B_{\mathcal{F}}$ and any $x,y$ in $G$, from the $\sigma$-stability of $h$ one has that $h(I_{a}x+I_{a^{c}}y)=I_{a}h(x)+I_{a^{c}}h(y)$. Then, 
	when $I_{a}x=I_{a}y$, $h(y)=I_{a}h(x)+I_{a^{c}}h(y)$, so $I_{a}h(x)=I_{a}h(y)$ holds, which amounts to saying that $h$ is local.
\end{proof}

\par Lemma \ref{lemm.5.2} below is a special case of the main result of \cite{Ponosov1987}, which is sufficient for our purpose.

\begin{lemma}[\cite{Ponosov1987}]\label{lemm.5.2}
	Let $X$ and $Y$ be two complete separable metric spaces, $(\Omega,\mathcal{F},P)$ a complete probability space and $h:L^{0}(\mathcal{F},X)\rightarrow L^{0}(\mathcal{F},Y)$ a local 
	and continuous-in-probability operator. Then there exists a Carath\'{e}odory function $f:\Omega\times X\rightarrow Y$ (namely,  $f(\cdot,x)$ is measurable for any $x\in X$ and $f(\omega,\cdot)$
	is continuous for any $\omega\in \Omega$) such that $h=h_{f}$. Here, $h_{f}:L^{0}(\mathcal{F},X)\rightarrow L^{0}(\mathcal{F},Y)$ is defined by $h_{f}(x)=$ the equivalence class of $f(\cdot,x^{0}(\cdot))$, 
	where $x^{0}$ is an arbitrarily chosen representative of $x$.
\end{lemma}

\par  In fact, Lemma \ref{lemm.5.2} has implied that $h$ is a.s. sequentially continuous, namely, $\{h(x_{n}),n\in \mathbb{N}\}$ converges a.s. to $h(x)$ 
whenever $\{x_{n},n\in \mathbb{N}\}$ converges a.s. to $x$.

\begin{lemma}\label{lemm.5.3}
	Let $U:\Omega\rightarrow 2^{\mathbb{R}^{d}}$ be the same as in Proposition \ref{prop.2.4}, then $G_{U}:=\{x\in L^0(\mathcal{F},\mathbb{R}^d):x$ has 
	a representative $x^{0}$ such that $x^{0}(\omega)\in U(\omega)~a.s.\}$ is an a.s. bounded $\mathcal{T}_{\varepsilon,\lambda}$-closed $L^{0}$-convex 
	subset of $L^0(\mathcal{F},\mathbb{R}^d)$.
\end{lemma}
\begin{proof}
	For clarity, we use $|\cdot|$ for the usual Euclidean norm on $\mathbb{R}^{d}$ and reserve $\|\cdot\|$ for the $L^{0}$-norm on $L^0(\mathcal{F},\mathbb{R}^d)$.
	We employ two methods to prove this lemma as follows. We only need to prove that $G_{U}$ is a.s. bounded since the others are obvious.
	
	\par \textbf{Method 1}. First, define a real-valued function $r:\Omega\rightarrow [0,+\infty)$ by $r(\omega)=\sup\{|x|:x\in U(\omega)\}$ for any $\omega\in \Omega$. Then $r$ is well-defined since each $U(\omega)$ is bounded. Now, let $G_{U}^{0}=\{x^{0}:x^{0}$ is an $\mathbb{R}^{d}$-valued $\mathcal{F}$-measurable function on $\Omega$
	such that $x^{0}(\omega)\in U(\omega)~a.s.\}$. Then $|x^{0}(\omega)|\leq r(\omega)$ a.s. for each $x^{0}\in G_{U}^{0}$. Further, let $\xi^{0}=ess.sup\{|x^{0}|:x^{0}\in G_{U}^{0}\}$,
	see, for example, \cite{Guo2010,GWXYC2025} for the notion of an essential supremum for a set of extended real-valued random variables. Then it still holds that $\xi^{0}(\omega)\leq r(\omega)$ a.s., so we can assume that $\xi^{0}$ is a real-valued random variable. Then $\bigvee\{\|x\|:x\in G_{U}\}\leq \xi$, where $\xi$ is the equivalence class of $\xi^{0}$, namely, $G_{U}$ is a.s. bounded.
	
	\par \textbf{Method 2}. Let $\hat{\mathcal{F}}$ be the completion of $\mathcal{F}$ with respect to $P$. Then the graph $Gr(U)\in \hat{\mathcal{F}}\otimes Bor(\mathbb{R}^{d})$ and hence by \cite[Theorem 4.2]{Wagner1977} there exists a sequence $\{x_{n}^{0},n\in \mathbb{N}\}$ of $\hat{\mathcal{F}}$-measurable selections of $U$ such that $U(\omega)=\overline{\{x_{n}^{0}(\omega):n\in \mathbb{N}\}}$ 
	for each $\omega\in \Omega$. Let $r^{0}:\Omega\rightarrow [0,+\infty)$ be the function defined by $r^{0}(\omega)=\sup\{|x_{n}^{0}(\omega)|:n\in \mathbb{N}\}$ for any $\omega\in \Omega$, then $r^{0}$ is 
	$\hat{\mathcal{F}}$-measurable and hence there exists an $\mathcal{F}$-measurable nonnegative random variable $r_{1}^{0}$ such that $r^{0}(\omega)=r_{1}^{0}(\omega)$ a.s. Further, let $r_{1}$ be the 
	equivalence class of $r_{1}^{0}$. Then it is clear that $\bigvee\{\|x\|:x\in G_{U}\}\leq r_{1}$; $G_{U}$ is again a.s. bounded.
\end{proof}

\begin{theorem}\label{thm.equivalent1}
	Proposition \ref{prop.2.3} is equivalent to Proposition \ref{prop.2.4}.
\end{theorem}
\begin{proof}
	Proposition \ref{prop.2.3} $\Rightarrow$ Proposition \ref{prop.2.4} is clear by Lemma \ref{lemm.5.1} and the fact that the $\mathcal{T}_{\varepsilon,\lambda}$-continuity of $h$ is exactly the 
	continuity-in-probability of $h$.
	
	\par For the reverse implication, for any $\xi\in L_{++}^{0}(\mathcal{F})$, let $T_{\xi}$ be the isomorphism of $L^0(\mathcal{F},\mathbb{R}^d)$ to itself, defined by $T_{\xi}(x)=\xi x$ for any $x\in L^0(\mathcal{F},\mathbb{R}^d)$. Then it is also clear that $T_{\xi}^{-1}=T_{\xi^{-1}}$. Now, let $G$ and $f$ be the same as assumed in Proposition \ref{prop.2.3}. Then there exists some $\xi\in L_{++}^{0}(\mathcal{F})$ such that $\xi^{-1}G\subseteq L^{0}(\mathcal{F},\bar{B}_{1}(\mathbb{R}^{d}))$, where $\bar{B}_{1}(\mathbb{R}^{d})=\{y\in \mathbb{R}^{d}:|y|\leq 1\}$ and $|\cdot|$ stands for the Euclidean norm on $\mathbb{R}^{d}$. For brevity, we denote $L^{0}(\mathcal{F},\bar{B}_{1}(\mathbb{R}^{d}))$ by $V$, then $V$ is an a.s. bounded, $\mathcal{T}_{\varepsilon,\lambda}$-closed and $L^{0}$-convex subset of $L^{0}(\mathcal{F},\mathbb{R}^{d})$. Further, let $P_{V,\xi^{-1}G}$ be the restriction of $P_{\xi^{-1}G}$ to $V$, where $P_{\xi^{-1}G}$ is the random projection of $L^{0}(\mathcal{F},\mathbb{R}^{d})$ onto $\xi^{-1}G$. Then $h:=T_{\xi^{-1}}\circ f\circ T_{\xi}\circ P_{V,\xi^{-1}G}$ is a $\sigma$-stable and $\mathcal{T}_{\varepsilon,\lambda}$-continuous mapping from $V$ to $V$, and hence $h$ is also a local and continuous-in-probability operator. Now, define a set-valued function $U:\Omega\rightarrow 2^{\mathbb{R}^{d}}$ by $U(\omega)=\bar{B}_{1}(\mathbb{R}^{d})$ for each $\omega\in \Omega$. Then $U$ not only satisfies the hypothesis of Proposition \ref{prop.2.4} but also is weakly measurable in the sense of \cite{Wagner1977}, and in particular $G_{U}$ is exactly $V$. Thus, there exists $v\in V$ such that $h(v)=v$ by Proposition \ref{prop.2.4}. Then $\xi v=f(\xi P_{V,\xi^{-1}G}(v))\in G$. Again, let $g\in G$ be such that $g=\xi v$. Then $v=\xi^{-1} g\in \xi^{-1}G$, and $P_{V,\xi^{-1}G}(v)=\xi^{-1}g$. Hence $f(\xi P_{V,\xi^{-1}G}(v))=f(g)=\xi v=g$, namely, $g$ is a fixed point of $f$ in $G$.  
\end{proof}

\begin{remark}\label{rmk.5.5}
	Let $V=L^{0}(\mathcal{F},\bar{B}_{1}(\mathbb{R}^{d}))$ and $h:V\rightarrow V$ be the same as obtained in the process of the proof of that Proposition \ref{prop.2.4} implies Proposition \ref{prop.2.3}.
	If $(\Omega,\mathcal{F},P)$ is a complete probability space, then by Lemma \ref{lemm.5.2} there exists a Carath\'{e}odory function $f:\Omega\times \bar{B}_{1}(\mathbb{R}^{d})\rightarrow \bar{B}_{1}(\mathbb{R}^{d})$
	such that $h(v)=h_{f}(v)$ for any $v\in L^{0}(\mathcal{F},\bar{B}_{1}(\mathbb{R}^{d}))$. This observation leads us to the following interesting (or even somewhat surprising) result.
\end{remark}

\begin{theorem}\label{thm.equivalent2}
	Both Proposition \ref{prop.2.3} and Proposition \ref{prop.2.4} are equivalent to either of the following two assertions:
	\begin{enumerate} 
		\item [(1)] Every $\mathcal{T}_{\varepsilon,\lambda}$-continuous $\sigma$-stable mapping $h$ from $L^{0}(\mathcal{F},\bar{B}_{1}(\mathbb{R}^{d}))$ to itself has a fixed point.
		\item [(2)] Every Carath\'{e}odory function function $f:(\Omega,\mathcal{F},P)\times \bar{B}_{1}(\mathbb{R}^{d})\rightarrow \bar{B}_{1}(\mathbb{R}^{d})$ admits a random fixed point, namely, there exists a random vector $v:(\Omega,\mathcal{F},P)\rightarrow \bar{B}_{1}(\mathbb{R}^{d})$ such that $f(\omega,v(\omega))=v(\omega)$ a.s..
	\end{enumerate}
\end{theorem}
\begin{proof}
	Theorem \ref{thm.equivalent1} has shown that Proposition \ref{prop.2.3} and Proposition \ref{prop.2.4} are equivalent, while the second part of the proof of Theorem \ref{thm.equivalent1} has also shown
	that Proposition \ref{prop.2.3} is equivalent to (1). Thus, we only need to prove that (1) and (2) are equivalent.
	
	\par $(1)\Rightarrow (2)$. Let $h=h_{f}$. Then there exists some $x\in L^{0}(\mathcal{F},\bar{B}_{1}(\mathbb{R}^{d}))$ such that $h_{f}(x)=x$, and hence taking $v$ to be an arbitrarily chosen representative
	of $x$ yields $f(\omega,v(\omega))=v(\omega)$ a.s., namely, $v$ satisfies (2).
	
	\par $(2)\Rightarrow (1)$. Let $\hat{\mathcal{F}}$ be the completion of $\mathcal{F}$ with respect to $P$ and further notice the fact that $L^{0}(\hat{\mathcal{F}},\bar{B}_{1}(\mathbb{R}^{d}))=L^{0}(\mathcal{F},\bar{B}_{1}(\mathbb{R}^{d}))$ although $\hat{\mathcal{F}}$ and $\mathcal{F}$ may differ. Then by Lemma \ref{lemm.5.2} there exists some Carath\'{e}odory function $f:(\Omega,\hat{\mathcal{F}},P)\times \bar{B}_{1}(\mathbb{R}^{d})\rightarrow \bar{B}_{1}(\mathbb{R}^{d})$ such that $h=h_{f}$. By (1) there exists some $\hat{\mathcal{F}}$-measurable $v:\Omega\rightarrow \bar{B}_{1}(\mathbb{R}^{d})$ such that $f(\omega,v(\omega))=v(\omega)$ a.s.. Let $x$ be the equivalence class of $v$. Then $h(x)=h_{f}(x)=x$, since $L^{0}(\hat{\mathcal{F}},\bar{B}_{1}(\mathbb{R}^{d}))=L^{0}(\mathcal{F},\bar{B}_{1}(\mathbb{R}^{d}))$, $x$ is also regarded as a fixed point of $h$ in $L^{0}(\mathcal{F},\bar{B}_{1}(\mathbb{R}^{d}))$.
\end{proof}

\begin{remark}\label{rmk.5.7}
	The original Brouwer fixed point theorem states that every continuous function $f$ from $\bar{B}_{1}(\mathbb{R}^{d})$ to itself has a fixed point, see \cite{Brouwer1912}. (2) of Theorem \ref{thm.equivalent2}
	can be aptly called the random version of the original Brouwer fixed point, which is equivalent to either of Proposition \ref{prop.2.3} and Proposition \ref{prop.2.4}. Thus, why we can obtain such a deep
	result as Theorem \ref{thm.equivalent2} may lie in the fact that Lemma \ref{lemm.5.2} due to Ponosov is surprisingly deep! Proposition \ref{prop.5.8} below should also be credited to Lemma \ref{lemm.5.2}.  
\end{remark}

\begin{proposition}\label{prop.5.8}
	Let $G$ be a $\mathcal{T}_{\varepsilon,\lambda}$-closed $L^{0}$-convex subset of $L^{0}(\mathcal{F},\mathbb{R}^{d})$ and $T:G\rightarrow L^{0}(\mathcal{F},\mathbb{R}^{d})$ be a $\sigma$-stable $\mathcal{T}_{\varepsilon,\lambda}$-continuous mapping. Then $T$ is a.s. sequentially continuous.
\end{proposition}
\begin{proof}
	Since elements of both $G$ and $L^{0}(\mathcal{F},\mathbb{R}^{d})$ are equivalence classes, we can, without loss of generality, assume that $(\Omega,\mathcal{F},P)$ is complete. Further, let $P_{G}$
	be the random projection of $L^{0}(\mathcal{F},\mathbb{R}^{d})$ onto $G$. Then $T\circ P_{G}:L^{0}(\mathcal{F},\mathbb{R}^{d})\rightarrow L^{0}(\mathcal{F},\mathbb{R}^{d})$ is still $\sigma$-stable
	and $\mathcal{T}_{\varepsilon,\lambda}$-continuous, and hence by Lemma \ref{lemm.5.2} there exists some Carath\'{e}odory function $f:(\Omega,\mathcal{F},P)\times \mathbb{R}^{d}\rightarrow \mathbb{R}^{d}$ such that 
	$T\circ P_{G}=h_{f}$, in particular $T(g)=h_{f}(g)$ for any $g\in G$, which thus implies that $T$ is a.s. sequentially continuous.
\end{proof}

\begin{remark}
	Proposition \ref{prop.5.8} has implied the equivalence of Proposition 3.1 of  \cite{DKKS2013} with either of Proposition \ref{prop.2.3} and Proposition \ref{prop.2.4}. 
\end{remark}

\section{The equivalence between the random Brouwer fixed point theorem and random Borsuk theorem}\label{sec.6}

\par Proposition \ref{prop.6.1} below shows that Theorem \ref{thm.Brouwer} is equivalent to a special case of it, namely, random Brouwer fixed point theorem for random closed balls, which will provide much convenience for our research into the random version of the classical Borsuk theorem.

\begin{proposition}\label{prop.6.1}
	The following three statements are equivalent:
	\begin{itemize}
		\item [(1)] Every $\mathcal{T}_{c}$-continuous $\sigma$-stable mapping from a $\mathcal{T}_{c}$-closed, $\sigma$-stable, a.s. bounded and $L^{0}$-convex subset $G$ of $L^{0}(\mathcal{F},\mathbb{R}^{d})$ to $G$ has a fixed point.
		\item [(2)] Every $\mathcal{T}_{c}$-continuous $\sigma$-stable mapping from a random closed ball $ \bar{B}(\theta,r)$ of $L^{0}(\mathcal{F},\mathbb{R}^{d})$ to $ \bar{B}(\theta,r)$ has a fixed point, where $r\in L^{0}_{++}(\mathcal{F})$ and $ \bar{B}(\theta,r)=\{x\in L^{0}(\mathcal{F},\mathbb{R}^{d}):\|x\|\leq r\}$.
		\item [(3)] Every $\mathcal{T}_{c}$-continuous $\sigma$-stable mapping from the random closed unit ball $ \bar{B}(\theta,1)$ of $L^{0}(\mathcal{F},\mathbb{R}^{d})$ to $ \bar{B}(\theta,1)$ has a fixed 
		point.
	\end{itemize} 
\end{proposition}
\begin{proof}
	$(1)\Rightarrow (2)$ is clear.

	\par $(2)\Rightarrow (1)$. Let $f:G\rightarrow G$ be a $\mathcal{T}_{c}$-continuous $\sigma$-stable mapping. Since $G$ is a.s. bounded, we can choose a sufficiently large $r\in L^{0}_{++}(\mathcal{F})$ such that $G\subseteq  \bar{B}(\theta,r)$. Further, since 
	$G$ is $\sigma$-stable and $\mathcal{T}_{c}$-closed, it follows immediately from \cite{Guo2010} that $G$ is also $\mathcal{T}_{\varepsilon,\lambda}$-closed. Then by Lemma \ref{lemm.projection}
	the random projection $P_{G}$ of $L^{0}(\mathcal{F},\mathbb{R}^{d})$ onto $G$ is well defined, $\sigma$-stable and $\mathcal{T}_{c}$-continuous. Now, let $P_{ \bar{B}(\theta,r),G}$ be the restriction
	of $P_{G}$ to $ \bar{B}(\theta,r)$. Then it is also $\sigma$-stable and $\mathcal{T}_{c}$-continuous since $ \bar{B}(\theta,r)$ is $\sigma$-stable. Thus, it is clear that $f\circ P_{ \bar{B}(\theta,r),G}$ is a 
	$\mathcal{T}_{c}$-continuous and $\sigma$-stable mapping from $ \bar{B}(\theta,r)$ to $ \bar{B}(\theta,r)$. By (2) it follows that $f\circ P_{ \bar{B}(\theta,r),G}$ has a fixed point $x$. Eventually, $f(x)=x$ since $x\in G$.

	\par  $(2)\Leftrightarrow (3)$ is clear since $T_{r}: \bar{B}(\theta,1)\rightarrow  \bar{B}(\theta,r)$, defined by $T_{r}(x)=rx$ for each $x\in  \bar{B}(\theta,1)$, is a  $\sigma$-stable $\mathcal{T}_{c}$-homeomorphism.
\end{proof}

\par In 1967, Borsuk \cite{Borsuk1967} introduced the notion of an $r$-image for a subset of $\mathbb{R}^{d}$. Now, according to the two topologies $\mathcal{T}_{c}$ and $\mathcal{T}_{\varepsilon,\lambda}$
of $L^{0}(\mathcal{F},\mathbb{R}^{d})$, we can introduce the corresponding two notions of a $\mathcal{T}_{c}$-$r$-image and a $\mathcal{T}_{\varepsilon,\lambda}$-$r$-image for a $\sigma$-stable subset of
$L^{0}(\mathcal{F},\mathbb{R}^{d})$, respectively, as follows.

\begin{definition}\label{defn.r-image}
	 Let $E$ and $F$ be two $\sigma$-stable subsets of $L^{0}(\mathcal{F},\mathbb{R}^{d})$. $E$ is called a $\mathcal{T}_{c}$-$r$-image (resp., a $\mathcal{T}_{\varepsilon,\lambda}$-$r$-image) of $F$ if 
	 there exist a $\mathcal{T}_{c}$-continuous (resp., $\mathcal{T}_{\varepsilon,\lambda}$-continuous) $\sigma$-stable mapping $h:F\rightarrow E$ and a $\mathcal{T}_{c}$-continuous 
	 (resp., $\mathcal{T}_{\varepsilon,\lambda}$-continuous) $\sigma$-stable mapping $g:E\rightarrow F$ such that $h\circ g$ is the identity on $E$. Such a mapping $h$ is called a $\mathcal{T}_{c}$-$r$-mapping 
	 (resp., a $\mathcal{T}_{\varepsilon,\lambda}$-$r$-mapping) of $F$ onto $E$. In the special case where $E\subseteq F$ and $g$ is the inclusion mapping, then $E$ is called a $\mathcal{T}_{c}$-retract
	 (resp., a $\mathcal{T}_{\varepsilon,\lambda}$-retract) of $F$ and $h$ is called a $\mathcal{T}_{c}$-retraction (resp., a $\mathcal{T}_{\varepsilon,\lambda}$-retraction). 
\end{definition}

\par Let $E$ and $F$ be the same as in Definition \ref{defn.r-image}. If $h:F\rightarrow E$ is a $\sigma$-stable and bijective mapping, then it is easy to check that $h^{-1}:E\rightarrow F$ is also $\sigma$-stable.
It is also clear that a $\sigma$-stable $\mathcal{T}_{c}$-homeomorphism (resp., $\mathcal{T}_{\varepsilon,\lambda}$-homeomorphism) of $F$ onto $E$ must be a $\mathcal{T}_{c}$-$r$-mapping 
(resp., a $\mathcal{T}_{\varepsilon,\lambda}$-$r$-mapping) of $F$ onto $E$. It follows immediately from \cite[Lemma 4.3]{GWXYC2025} that a $\mathcal{T}_{\varepsilon,\lambda}$-$r$-mapping 
(resp., a $\mathcal{T}_{\varepsilon,\lambda}$-retraction) of $F$ onto $E$ must be a $\mathcal{T}_{c}$-$r$-mapping (resp., a $\mathcal{T}_{c}$-retraction) of $F$ onto $E$.

\par Proposition \ref{prop.6.3} below is a random version of \cite[Theorem 6.9]{Bor1985}, which is essentially a reformulation of Theorem \ref{thm.Brouwer} in a more general form.

\begin{proposition}\label{prop.6.3}
	Let $E$ be a a $\mathcal{T}_{c}$-$r$-image of an a.s. bounded, $\mathcal{T}_{c}$-closed, $\sigma$-stable and $L^{0}$-convex subset $F$ of $L^{0}(\mathcal{F},\mathbb{R}^{d})$ and $f:E\rightarrow E$ be 
	$\sigma$-stable and $\mathcal{T}_{c}$-continuous. Then $f$ has a fixed point.
\end{proposition}
\begin{proof}
	Let $g:E\rightarrow F$ and $h:F\rightarrow E$ be the same as in Definition \ref{defn.r-image}. Then $g\circ f\circ h:F\rightarrow F$ has a fixed point $z$ by Theorem \ref{thm.Brouwer}. Let $x=h(z)$. Then
	$$f(x)=(h\circ g)(f(x))=(h\circ g\circ f\circ h)(z)=h(z)=x.$$
\end{proof}

\par Theorem \ref{thm.6.4} below can be aptly called a random Borsuk theorem, which is a random version of the classical Borsuk theorem \cite{Borsuk1967} (see also \cite[Theorem 6.12]{Bor1985}). Here, we would like to remind the reader that $L^{0}(\mathcal{F},\bar{B}_{1}(\mathbb{R}^{d}))$ in the remainder of this section is the very same thing as the random closed unit ball $\bar{B}(\theta,1):=\{x\in L^{0}(\mathcal{F},\mathbb{R}^{d}):\|x\|\leq 1\}$.


\begin{theorem}\label{thm.6.4}
	$L^{0}(\mathcal{F},\partial \bar{B}_{1}(\mathbb{R}^{d}))$ is not a $\mathcal{T}_{c}$-$r$-image of $L^{0}(\mathcal{F},\bar{B}_{1}(\mathbb{R}^{d}))$, where, as usual,  $\bar{B}_{1}(\mathbb{R}^{d})=\{y\in \mathbb{R}^{d}:|y|\leq 1\}$, $\partial \bar{B}_{1}(\mathbb{R}^{d})=\{y\in \mathbb{R}^{d}:|y|=1\}$ and $|\cdot|$ denotes the Euclidean norm on $\mathbb{R}^{d}$. 
	In particular, $L^{0}(\mathcal{F},\partial \bar{B}_{1}(\mathbb{R}^{d}))$ is not a $\mathcal{T}_{c}$-retraction of $L^{0}(\mathcal{F},\bar{B}_{1}(\mathbb{R}^{d}))$.
\end{theorem}
\begin{proof}
	It is easy to see that both $L^{0}(\mathcal{F},\partial \bar{B}_{1}(\mathbb{R}^{d}))$ and $L^{0}(\mathcal{F},\bar{B}_{1}(\mathbb{R}^{d}))$ are $\sigma$-stable. If $L^{0}(\mathcal{F},\partial \bar{B}_{1}(\mathbb{R}^{d}))$ is a $\mathcal{T}_{c}$-$r$-image 
	of $L^{0}(\mathcal{F},\bar{B}_{1}(\mathbb{R}^{d}))$, then there are a $\sigma$-stable $\mathcal{T}_{c}$-continuous mapping $h:L^{0}(\mathcal{F},\bar{B}_{1}(\mathbb{R}^{d}))\rightarrow L^{0}(\mathcal{F},\partial \bar{B}_{1}(\mathbb{R}^{d}))$ and 
	a $\sigma$-stable $\mathcal{T}_{c}$-continuous mapping $g:L^{0}(\mathcal{F},\partial \bar{B}_{1}(\mathbb{R}^{d}))\rightarrow L^{0}(\mathcal{F},\bar{B}_{1}(\mathbb{R}^{d}))$ such that $h\circ g$ is the identity on $L^{0}(\mathcal{F},\partial \bar{B}_{1}(\mathbb{R}^{d}))$.
	Define $f:L^{0}(\mathcal{F}, \bar{B}_{1}(\mathbb{R}^{d}))\rightarrow L^{0}(\mathcal{F},\bar{B}_{1}(\mathbb{R}^{d}))$ by $f(x)=g(-h(x))$ for any $x\in L^{0}(\mathcal{F}, \bar{B}_{1}(\mathbb{R}^{d}))$, then $f$ is a $\sigma$-stable and $\mathcal{T}_{c}$-continuous
	mapping, and hence there exists $z\in L^{0}(\mathcal{F}, \bar{B}_{1}(\mathbb{R}^{d}))$ by Theorem \ref{thm.Brouwer} such that $f(z)=z$, that is to say, $z=g(-h(z))$, so $h(z)=(h\circ g)(-h(z))=-h(z)$. Thus 
	$h(z)=0\in L^{0}(\mathcal{F},\partial \bar{B}_{1}(\mathbb{R}^{d}))$, a contradiction.
\end{proof}

\par Since a $\mathcal{T}_{\varepsilon,\lambda}$-$r$-image and a $\mathcal{T}_{\varepsilon,\lambda}$-retraction of $L^{0}(\mathcal{F},\bar{B}_{1}(\mathbb{R}^{d}))$ must be a $\mathcal{T}_{c}$-$r$-image and a $\mathcal{T}_{c}$-retraction of $L^{0}(\mathcal{F},\bar{B}_{1}(\mathbb{R}^{d}))$, respectively, we can immediately obtain Corollary \ref{coro.6.5} below, which includes \cite[Lemma 3.5]{DKKS2013} as a special case.

\begin{corollary}\label{coro.6.5}
 $L^{0}(\mathcal{F},\partial \bar{B}_{1}(\mathbb{R}^{d}))$ is not a $\mathcal{T}_{\varepsilon,\lambda}$-$r$-image of $L^{0}(\mathcal{F},\bar{B}_{1}(\mathbb{R}^{d}))$.
 In particular, $L^{0}(\mathcal{F},\partial \bar{B}_{1}(\mathbb{R}^{d}))$ is not a $\mathcal{T}_{\varepsilon,\lambda}$-retraction of $L^{0}(\mathcal{F},\bar{B}_{1}(\mathbb{R}^{d}))$.
\end{corollary}

\begin{theorem}\label{thm.6.5}
	 Theorem \ref{thm.6.4} is equivalent to the random Brouwer fixed point theorem for the random closed unite ball $L^{0}(\mathcal{F}, \bar{B}_{1}(\mathbb{R}^{d}))$ of $L^{0}(\mathcal{F},\mathbb{R}^{d})$ (and hence also equivalent to Theorem \ref{thm.Brouwer} 
	 by Proposition \ref{prop.6.1}).
\end{theorem}
\begin{proof}
	The proof of Theorem \ref{thm.6.4} has shown that the random Brouwer fixed point theorem implies Theorem \ref{thm.6.4}. We only need to prove that Theorem \ref{thm.6.4} implies the random Brouwer
	fixed point theorem for the random closed unit ball of $L^{0}(\mathcal{F},\mathbb{R}^{d})$ as follows.

	\par Suppose that there exists a $\sigma$-stable and $\mathcal{T}_{c}$-continuous mapping 
	$f:L^{0}(\mathcal{F},\bar{B}_{1}(\mathbb{R}^{d}))\rightarrow L^{0}(\mathcal{F},\bar{B}_{1}(\mathbb{R}^{d}))$ such that $f$ has no fixed point. Define $e:L^{0}(\mathcal{F},\bar{B}_{1}(\mathbb{R}^{d}))\rightarrow L^{0}(\mathcal{F})$ by $e(x)=\|f(x)-x\|$ for any 
	$x\in L^{0}(\mathcal{F},\bar{B}_{1}(\mathbb{R}^{d}))$, and denote $\eta=\bigwedge\{e(x):x\in L^{0}(\mathcal{F},\bar{B}_{1}(\mathbb{R}^{d}))\}$. Then by \cite[Lemma 3.4]{GWXYC2025} there exists $x_{0}\in L^{0}(\mathcal{F},\bar{B}_{1}(\mathbb{R}^{d}))$
	such that $\eta=\|f(x_{0})-x_{0}\|$ since $L^{0}(\mathcal{F},\bar{B}_{1}(\mathbb{R}^{d}))\}$ is random sequentially compact and $\sigma$-stable and since $e$ is $\sigma$-stable and random sequentially continuous. 
	Then we have $\eta>0$, since otherwise, $\eta=0$ yields $f(x_{0})=x_{0}$, which contradicts the assumption that $f$ has no fixed point.

	\par Now, let $A\in \mathcal{F}$ be an arbitrary chosen representative of $(\eta>0)$. Then $P(A)>0$. Since $\theta\in L^{0}(\mathcal{F},\bar{B}_{1}(\mathbb{R}^{d}))$
    and $f:L^{0}(\mathcal{F},\bar{B}_{1}(\mathbb{R}^{d})) \rightarrow L^{0}(\mathcal{F},\bar{B}_{1}(\mathbb{R}^{d}))$ is $\sigma$-stable, we know that $f$ has the local property, namely, 
	$\tilde{I}_{C}f(x)=\tilde{I}_{C}f(\tilde{I}_{C}x)$ for any $C\in \mathcal{F}$ and $x\in L^{0}(\mathcal{F},\bar{B}_{1}(\mathbb{R}^{d}))$: in fact, since 
	$$f(\tilde{I}_{C}x)=f(\tilde{I}_{C}x+\tilde{I}_{C^c}\theta)=\tilde{I}_{C}f(x)+\tilde{I}_{C^c}f(\theta),$$
	we have $\tilde{I}_{C}f(\tilde{I}_{C}x)= \tilde{I}_{C}f(x)$.
	In particular, $\tilde{I}_{A}f(\tilde{I}_{A}x)= \tilde{I}_{A}f(x)$ for any $x\in L^{0}(\mathcal{F},\bar{B}_{1}(\mathbb{R}^{d}))$. Further, let 
	$$G_{A}=\tilde{I}_{A} L^{0}(\mathcal{F},\bar{B}_{1}(\mathbb{R}^{d})):=\{\tilde{I}_{A}x: x\in L^{0}(\mathcal{F},\bar{B}_{1}(\mathbb{R}^{d}))\}$$ 
	and define $f_{A}:G_{A}\rightarrow G_{A}$ by 
	$f_{A}(\tilde{I}_{A}x)= \tilde{I}_{A}f(x)$ for any $x\in L^{0}(\mathcal{F},\bar{B}_{1}(\mathbb{R}^{d}))$. Then $f_{A}$ is well defined. It is clear that $G_{A}$ is a $\sigma$-stable, $\mathcal{T}_{c}$-closed, a.s. bounded and $L^{0}(\mathcal{F})$-convex subset of $L^{0}(\mathcal{F},R^{d})$ and 
	$f_{A}$ is also $\sigma$-stable and $\mathcal{T}_{c}$-continuous.

	\par Let $(A,A\cap \mathcal{F},P_{A})$ be the probability space with $P_{A}:A\cap \mathcal{F} \rightarrow [0,1]$ defined by 
	$P_{A}(A\cap C)=\frac{P(A\cap C)}{P(A)}$ for any $C\in \mathcal{F}$. Further, identifying $G_{A}$ with $L^{0}(\mathcal{F}_{A},\bar{B}_{1}(\mathbb{R}^{d}))$, 
	where $\mathcal{F}_{A}=A\cap \mathcal{F}$. Then $f_{A}$ can be also identified with a mapping from $L^{0}(\mathcal{F}_{A},\bar{B}_{1}(\mathbb{R}^{d}))$ to 
	$L^{0}(\mathcal{F}_{A},\bar{B}_{1}(\mathbb{R}^{d}))$. Again observing $\|f(x)-x\|>0$ on $A$ for any $x\in L^{0}(\mathcal{F},\bar{B}_{1}(\mathbb{R}^{d}))$, namely, 
	$\|f_{A}(x)-x\|>0$ on $A$ for each $x\in G_{A}$, we can, without loss of generality, assume that $A=\Omega$. Then 
	$\|f(x)-x\|>0$ on $\Omega$ for any $x\in L^{0}(\mathcal{F},\bar{B}_{1}(\mathbb{R}^{d}))$ (since otherwise, we can consider $f_{A}$ and $L^{0}(\mathcal{F}_{A},\bar{B}_{1}(\mathbb{R}^{d}))$
	to replace $f$ and $L^{0}(\mathcal{F},\bar{B}_{1}(\mathbb{R}^{d}))$).

	\par Finally, let $L_{x}=\{\lambda\in L^{0}(\mathcal{F}):\|x+\lambda(f(x)-x)\|=1\}$ for any $x\in L^{0}(\mathcal{F},\bar{B}_{1}(\mathbb{R}^{d}))$ 
	and $\lambda(x)=Min~L_{x}:=$ the smallest element of $L_{x}$. In fact, $L_{x}$ has only the two elements, satisfying the 
	equation 
	$$\|x\|^{2}+2\lambda \langle x,f(x)-x\rangle+\lambda^{2}\|f(x)-x\|^{2}-1=0,$$
	where $\langle x,y\rangle$ stands for the $L^{0}$-inner product of $x$ and $y$
	in $L^{0}(\mathcal{F},\mathbb{R}^{d})$. Similar to the root-finding formula of a quadratic equation with one variable, one can have 
	$$\lambda(x)=\frac{ \langle x,x-f(x)\rangle-\sqrt{\langle x,x-f(x)\rangle^{2}+\|x-f(x)\|^{2} (1-\|x\|)^{2}}}{\|f(x)-x\|^{2}}.$$
	This define a mapping $\lambda:L^{0}(\mathcal{F},\bar{B}_{1}(\mathbb{R}^{d})) \rightarrow L^{0}(\mathcal{F})$ and it is not difficult to verify 
	that $\lambda$ is still $\sigma$-stable and $\mathcal{T}_{c}$-continuous. By noting that $\langle x,x-f(x)\rangle=\|x\|^{2}-\langle x,f(x)\rangle$ 
	and the random Schwartz inequality: $|\langle x,f(x)\rangle|\leq \|x\|\|f(x)\|$, we have 
	$$\langle x,x-f(x)\rangle\geq \|x\|(1-\|f(x)\|)\geq 0$$ 
	for any $x\in L^{0}(\mathcal{F},\bar{B}_{1}(\mathbb{R}^{d}))$. In particular, when $x\in L^{0}(\mathcal{F},\partial \bar{B}_{1}(\mathbb{R}^{d}))$, namely, when $\|x\|=1$, one has $\lambda(x)=0$. Again, define $h:L^{0}(\mathcal{F},\bar{B}_{1}(\mathbb{R}^{d}))\rightarrow  L^{0}(\mathcal{F},\partial \bar{B}_{1}(\mathbb{R}^{d}))$ by $h(x)=x+\lambda(x)(f(x)-x)$. Then $h$ is a $\mathcal{T}_{c}$-retraction of $L^{0}(\mathcal{F},\bar{B}_{1}(\mathbb{R}^{d}))$ onto $L^{0}(\mathcal{F},\partial \bar{B}_{1}(\mathbb{R}^{d}))$,
	which contradicts Theorem \ref{thm.6.4}.	
\end{proof}

    \begin{remark}
    	The map $h$ obtained in the final part of the proof of Theorem \ref{thm.6.5} is a $\mathcal{T}_{c}$-retraction. Actually, we have obtained an 
    	interesting conclusion, namely, the random Brouwer fixed point theorem is equivalent to the fact that $L^{0}(\mathcal{F},\partial \bar{B}_{1}(\mathbb{R}^{d}))$
    	is not a $\mathcal{T}_{c}$-retraction of $L^{0}(\mathcal{F},\bar{B}_{1}(\mathbb{R}^{d}))$!
    \end{remark}

    \begin{remark}
    	By a completely similar method as used in Theorem \ref{thm.6.5}, one can easily prove that Corollary \ref{coro.6.5} is equivalent to Proposition \ref{prop.2.3}. 
    \end{remark}

\section{Commentaries on the state of the Schauder conjecture}\label{sec.7}

	Let $X$ be a Hausdorff linear topological space and $C$ be a compact convex subset of $X$. Then the famous
	Schauder conjecture asks if every continuous mapping from $C$ to $C$ has a fixed point. When $X$ is locally convex,
	the answer is affirmative, that is the famous Tychonoff fixed point theorem \cite{Ty1935}. When $X$ is not locally convex, 
	Cauty \cite{Cauty2007} and Ennassik and Taoudi \cite{ET2021} claimed that they had solved the conjecture in 
	an affirmative way, however Ma\'{n}ka \cite{Manka} and Yu \cite{Yu2024} pointed out that both Cauty \cite{Cauty2007} and Ennassik and Taoudi \cite{ET2021} contained a gap, namely, the Schauder conjecture is still open. Although the works of this paper and \cite{GWXYC2025} also involve the fixed point problems in locally noconvex spaces, for example, an $RN$ module endowed with the $(\varepsilon,\lambda)$-topology is just a metrizable locally noconvex linear topological space 
	in general, our works are not affected by the results in \cite{Cauty2007} and \cite{ET2021} since our works depend on the $\sigma$-stability of the $\mathcal{T}_{\varepsilon,\lambda}$-continuous mappings and the theory of random sequentially 
	compact $L^{0}$-convex sets rather than any works in \cite{Cauty2007} and \cite{ET2021}. Besides, the main challenge from \cite{GWXYC2025} 
	up to this paper lies in overcoming noncompactness since a random sequentially compact set is generally noncompact, in fact, our works also tread a new path in the study of the topological fixed point theory in $RN$ modules.

\vspace{3mm}

\noindent\textbf{Acknowledgements} The authors would like to thank Professors Hong-Kun Xu, Lixin Cheng and George Xianzhi Yuan for their valuable suggestions on this work.  \\






\bibliographystyle{amsplain}

\end{document}